\documentclass[11pt,letterpaper,reqno,oneside]{amsart}
\usepackage[letterpaper, includehead, includefoot,top=25mm, bottom=25mm, left=25mm, right=25mm,marginparwidth=22mm, marginparsep=4mm]{geometry}
\usepackage{times}
\usepackage{comment}
\usepackage{amsfonts}
\usepackage{amsmath}
\usepackage{amssymb}
\usepackage{amsthm}
\usepackage{thmtools}
\usepackage[T1]{fontenc}
\usepackage{enumerate}
\usepackage{booktabs}
\usepackage{array}
\usepackage{caption}
\usepackage{graphicx}

\usepackage{xcolor}
\usepackage{hyperref}
\hypersetup{
    colorlinks=true,
    linkcolor=blue,
    citecolor=blue,
    filecolor=blue,
    urlcolor=blue
}

\usepackage[nameinlink,capitalize,noabbrev]{cleveref}

\declaretheorem[name=Theorem,numberwithin=section]{theorem}
\declaretheorem[sibling=theorem,name=Corollary]{corollary}
\declaretheorem[sibling=theorem,name=Lemma]{lemma}
\declaretheorem[sibling=theorem,name=Proposition]{proposition}
\declaretheorem[sibling=theorem,name=Remark,style=remark]{remark}
\declaretheorem[sibling=theorem,name=Definition,style=definition]{definition}
\declaretheorem[sibling=theorem,name=Notation,style=definition]{notation}
\declaretheorem[sibling=theorem,name=Example,style=definition]{example}

\newtheorem{introthm}{Theorem}

\newcommand{\PP}[0]{\ensuremath{\mathbf{P}}}
\newcommand{\CC}[0]{\ensuremath{\mathbf{C}}}
\newcommand{\ZZ}[0]{\ensuremath{\mathbf{Z}}}
\newcommand{\QQ}[0]{\ensuremath{\mathbf{Q}}}
\newcommand{\RR}[0]{\ensuremath{\mathbf{R}}}
\renewcommand{\AA}[0]{\ensuremath{\mathbf{A}}}
\newcommand{\cE}[0]{\ensuremath{\mathcal{E}}}
\newcommand{\cO}[0]{\ensuremath{\mathcal{O}}}
\newcommand{\cA}[0]{\ensuremath{\mathcal{A}}}

\newcommand{\Aut}[0]{\ensuremath{\operatorname{Aut}}}
\newcommand{\PGL}[0]{\ensuremath{\operatorname{PGL}}}
\newcommand{\GL}[0]{\ensuremath{\operatorname{GL}}}
\newcommand{\Id}[0]{\ensuremath{\operatorname{Id}}}

\newcommand{\Proj}[0]{\ensuremath{\operatorname{Proj}}}
\newcommand{\diag}[0]{\ensuremath{\operatorname{diag}}}

\newcommand{\Res}[0]{\ensuremath{\operatorname{Res}}}
\newcommand{\NS}[0]{\ensuremath{\operatorname{NS}}}
\newcommand{\End}[0]{\ensuremath{\operatorname{End}}}
\newcommand{\Hdg}[0]{\ensuremath{\mathrm{Hdg}}}
\newcommand{\disc}[0]{\ensuremath{\operatorname{disc}}}
\newcommand{\Tr}[0]{\ensuremath{\operatorname{Tr}}}
\newcommand{\Nm}[0]{\ensuremath{\operatorname{N}}}
\newcommand{\Ort}[0]{\ensuremath{\mathrm{O}}}
\newcommand{\prim}[0]{\ensuremath{\mathrm{prim}}}
\newcommand{\tr}[0]{\ensuremath{\mathrm{tr}}}
\newcommand{\alg}[0]{\ensuremath{\mathrm{alg}}}
\newcommand{\hyp}[1]{\ensuremath{\mathsf{H}(#1)}}

\newcommand{\PW}[0]{\ensuremath{\PP(1,2,2,3,3)}}
\newcommand{\Gm}[0]{\ensuremath{\CC^{\times}}}

\usepackage{todonotes}

\begin{document}
%----------------------------------------------------------------------

\title[Automorphisms and Hodge isometries of surfaces with $K^2=p_g=1$]{Automorphisms and Hodge isometries\\ of surfaces with $K^2=p_g=1$}

\author[J. Duque Franco]{Jorge Duque Franco}
\address{Dirección de Investigaci\'on, Vicerrector\'ia Acad\'emica and Instituto de Matem\'aticas, Universidad de Talca,
  Casilla 721, Talca, Chile}
\email{jorge.duque@utalca.cl}

\author[A. Liendo]{Alvaro Liendo}
\address{Instituto de Matem\'aticas, Universidad de Talca, Casilla 721, Talca, Chile}
\email{aliendo@utalca.cl}

\author[P. Montero]{Pedro Montero}
\address{Departamento de Matem\'atica, Universidad T\'ecnica Federico Santa Mar\'ia, Avenida Espa\~na 1680, Valpara\'iso, Chile}
\email{pedro.montero@usm.cl}

\date{\today}

\thanks{{\it 2020 Mathematics Subject Classification}: 14J50, 14J29, 14C30, 14L30.\\
  \mbox{\hspace{11pt}}{\it Key words}: surface of general type, automorphism group, canonical model, weighted complete intersection, Hodge structure, Torelli problem.\\
  \mbox{\hspace{11pt}}The second and third authors were partially supported by ANID via Fondecyt Regular projects 1240101 and 1231214 and Fondecyt Exploraci\'on project 13250049.}

\begin{abstract}
Let $S$ be a minimal complex surface of general type with $K_S^2=p_g(S)=1$. By a theorem of Catanese, the canonical model of $S$ is a complete intersection of two sextics in $\PW$ whose equations admit a normal form, and the global Torelli theorem fails for these surfaces. We describe $\Aut(S)$ in terms of the coefficients of the normal form, prove that $|\Aut(S)|\leq72$, and classify the surfaces attaining the bound: two surfaces with singular canonical model and a one-parameter family $X_r$ of smooth surfaces. For the surfaces $X_r$ we compute a basis of the primitive cohomology consisting of eigenvectors of a subgroup $\ZZ/6\ZZ\times\ZZ/6\ZZ$ of $\Aut(X_r)$, given by residues of explicit polynomials, together with the intersection form in this basis. We then show that the natural map from $\Aut(X_r)$ to the group of Hodge isometries of the primitive lattice modulo $\pm1$ is injective but has image of index at least $331\,776$, so that the Torelli principle for automorphisms fails for these surfaces.
\end{abstract}
\maketitle

%\tableofcontents

%----------------------------------------------------------------------
\section*{Introduction}
%----------------------------------------------------------------------

The Torelli problem asks whether a smooth projective variety is determined by its polarized Hodge structure. For minimal surfaces of general type with $K^2=p_g=1$, the subject of this paper, the answer is negative: by a theorem of Catanese, the period map of these surfaces has degree at least two \cite[Theorem~3.3]{Cat80}, so they give a counterexample to the global Torelli problem, see also \cite{Tod80} for counterexamples with $p_g=1$, $q=0$ and $2\leq K^2\leq8$. A related question is to what extent the automorphism group of a variety is determined by its polarized Hodge structure. For a K3 surface $X$ the answer is given by the strong Torelli theorem: an isometry of $\operatorname{H}^2(X,\ZZ)$ preserving the Hodge decomposition and the ample cone is induced by a unique automorphism of $X$ \cite{pyatetskii1971torelli,burns1975torelli}. For several Klein hypersurfaces, the automorphism group of the hypersurface and the automorphism group of its polarized Hodge structure coincide up to the involution $\pm1$ \cite{GALMVL24}, as predicted by a Torelli principle for automorphisms. In this paper we study this question for the surfaces with $K^2=p_g=1$ whose automorphism group is as large as possible. This requires first classifying these surfaces, which is the content of the first part of the paper.

The automorphism group of a minimal complex surface of general type $S$ is finite, and by a theorem of Xiao its order is at most $42^2K_S^2$ \cite{Xia94,Xia95}, the analogue of the classical bound $84(g-1)$ of Hurwitz for curves of genus $g\geq2$ \cite{Hur93}, which is attained by the Klein quartic \cite{Klein}. For $K_S^2=1$, Noether's inequality $K_S^2\geq2p_g(S)-4$ \cite[Theorem~VII.3.1]{BHPV04} gives $p_g(S)\leq2$. If $p_g(S)=2$, the canonical model of $S$ is a hypersurface of degree ten in $\PP(1,1,2,5)$, and Wen \cite{Wen21} and Zhao \cite{Zha26} proved that $|\Aut(S)|\leq200$ in this case. Throughout this paper we work over the field $\CC$ of complex numbers, $S$ is a minimal surface of general type with $K_S^2=p_g(S)=1$ and $X$ is its canonical model, a normal surface with at most rational double points. Since the minimal resolution $S\to X$ is unique, $\Aut(S)=\Aut(X)$. By a theorem of Catanese \cite{Cat79,Cat80}, $X$ is a complete intersection of two sextics in the weighted projective space $\PW$, contained in its smooth locus, and its equations can be put in the normal form
\[f=z_3^2+x_0z_4\,a(y)+F(y),\qquad g=z_4^2+x_0z_3\,b(y)+G(y).\]
Here $x_0,x_1,x_2,z_3,z_4$ are the coordinates of $\PW$, of weights $1,2,2,3,3$, we write
\[y=(y_0,y_1,y_2)=(x_0^2,x_1,x_2)\]
for the coordinates of $\PP(2,2,2)\simeq\PP^2$, and $a,b$ are linear forms and $F,G$ cubic forms in $y$, see \cref{thm:catanese}. In \cref{prop:autX} we describe $\Aut(X)$ in terms of the normal form: after a suitable choice of coordinates, every automorphism of $X$ is induced by a graded automorphism
\[x_0\mapsto x_0,\qquad (y_1,y_2)\mapsto A(y_1,y_2),\qquad z_3\mapsto\mu z_3,\qquad z_4\mapsto\nu z_4,\]
with $A\in\GL_2$ and $\mu,\nu\in\Gm$, or by such an automorphism followed by the exchange of $z_3$ and $z_4$, and the triples $(A,\mu,\nu)$ that occur are the solutions of explicit equations in $a,b,F,G$. This description reduces the computation of $|\Aut(X)|$ to a case analysis, which gives the following.

\begin{introthm}
\label{thm:A}
Let $S$ be a minimal complex surface of general type with $K_S^2=p_g(S)=1$. Then $|\Aut(S)|\leq 72$. Moreover, equality holds if and only if $S$ is the minimal resolution of one of the following complete intersections in $\PW$:
\[X_{2D_4}\colon\quad z_3^2+x_0^6+x_1^3+x_2^3=z_4^2+x_1^3-x_2^3=0,\]
\[X_{3A_2}\colon\quad z_3^2+x_0^6+x_1^3-x_2^3=z_4^2+x_0^6+x_1^3+x_2^3=0,\]
\[X_r\colon\quad z_3^2+rx_0^6+x_1^3-r^2x_2^3=z_4^2+x_0^6+x_1^3-x_2^3=0,\qquad r\in\Gm\smallsetminus\{\pm1\}.\]
Here, $X_{2D_4}$ has two rational double points of type $D_4$ and $\Aut(X_{2D_4})\simeq((\ZZ/3\ZZ)^2\rtimes\ZZ/4\ZZ)\times\ZZ/2\ZZ$. The surface $X_{3A_2}$ has three rational double points of type $A_2$ and $\Aut(X_{3A_2})\simeq(\ZZ/3\ZZ)^2\times \mathbf{D}_4$, where $\mathbf{D}_4$ is the dihedral group of order eight. Every $X_r$ is smooth, with $\Aut(X_r)\simeq(\ZZ/6\ZZ\times\ZZ/6\ZZ)\rtimes\ZZ/2\ZZ$, the group of order two exchanging the two factors. The only isomorphisms among these surfaces are $X_r\simeq X_{r^{-1}}$.
\end{introthm}

The bound of \cref{thm:A} is substantially sharper than Xiao's general bound $1764$ and than the bound $200$ for $p_g=2$. The reason is that the binary cubic forms $C=F(0,y_1,y_2)$ and $D=G(0,y_1,y_2)$ are coprime, so that $\Aut(X)$ acts on a set of at most six points of $\PP^1$ preserving a partition of it into two subsets of three points, and a projective transformation of $\PP^1$ fixing three points is the identity. This argument fails only when $C$ and $D$ are both cubes of linear forms. For these surfaces the order of the group is controlled by a lattice of characters and we prove that $|\Aut(X)|\leq64$ (\cref{prop:degenerate}). 

We now return to the Torelli principle for automorphisms. The extremal surfaces of \cref{thm:A} satisfy $a=b=0$, so they are bidouble covers of $\PP^2$ branched along the cubics $F=0$, $G=0$ and the line $y_0=0$ \cite[Proposition~2.2]{Cat80}, and on the locus of bidouble covers the differential of the period map is not injective \cite[Theorem~3.1]{Cat80}. In the second part of the paper we compute the Hodge structure of the smooth surfaces $X_r$ and we compare $\Aut(X_r)$ with the automorphism group of the polarized Hodge structure on $\operatorname{H}^2(X_r,\ZZ)$. Write $\Lambda=\operatorname{H}^2(X_r,\ZZ)$, a unimodular lattice of signature $(3,18)$, and $L=K^\perp\subset\Lambda$, which is even unimodular of signature $(2,18)$. Since $K^2=1$ we have $\Lambda=\ZZ K\oplus L$, so the Hodge isometries of $\Lambda$ fixing $K$ are exactly the isometries of $L$ preserving the Hodge decomposition of $L\otimes\CC$. We write $\Ort_\Hdg(L)$ for this group. Every automorphism of $X_r$ fixes $K$ and induces an element of $\Ort_\Hdg(L)$.

\begin{introthm}
\label{thm:B}
Let $r\in\Gm\smallsetminus\{\pm1\}$. The N\'eron--Severi
lattice of $X_r$ contains a root system of type
$A_2\sqcup A_2\sqcup D_4\sqcup D_4$ orthogonal to $K$,
whose Weyl group $W$, of order $6^2\cdot192^2$, acts on $L$
by Hodge isometries fixing $\operatorname{H}^{2,0}(X_r)$
pointwise. Consequently, $
|\Ort_\Hdg(L)|\geq36\,|W|=47\,775\,744.
$ The natural maps
\[
\Aut(X_r)\longrightarrow\Ort_\Hdg(L),
\qquad
\Aut(X_r)\longrightarrow\Ort_\Hdg(L)/\{\pm1\}
\]
are injective, and their images have index at least
$663\,552$ and $331\,776$, respectively.
\end{introthm}

In contrast with the situation for Klein hypersurfaces \cite{GALMVL24}, \Cref{thm:B} shows that the Torelli principle for automorphisms fails for the surfaces $X_r$. The reason is that the primitive cohomology of $X_r$ contains many algebraic classes: $L$ contains a negative definite sublattice of rank $16$ consisting of classes of type $(1,1)$, every isometry of $L$ acting trivially on its orthogonal complement is a Hodge isometry, and few of these isometries are induced by automorphisms. For very general $r$ we show moreover that the Hodge structure on the transcendental lattice determines exactly the quotient $\ZZ/6\ZZ$ of $\Aut(X_r)$ (\cref{prop:kernel}), so that the discrepancy of \cref{thm:B} is entirely supported on the N\'eron--Severi lattice. 

The proof of \cref{thm:B} occupies \cref{sec:griffiths,sec:hodge}. As in \cite{GALMVL24}, we first compute a basis of $\operatorname{H}^2(X_r,\CC)_\prim$ consisting of eigenvectors of the subgroup $G=\ZZ/6\ZZ\times\ZZ/6\ZZ$ of $\Aut(X_r)$ and formed by residues of explicit polynomials, most of them monomials, the analogue of the Griffiths basis of the primitive cohomology of a hypersurface \cite{gr69}. It is obtained from Mavlyutov's description of the Hodge components of a quasi-smooth complete intersection in a toric variety in terms of a Jacobian ring \cite{Mav99}, and we compute the intersection form in this basis using a toric version of the Carlson--Griffiths formula \cite{CG80} due to Villaflor Loyola \cite{VL24} (\cref{thm:gram}). The eigenspace decomposition gives a decomposition of $L\otimes\QQ$ into eight rational pieces, seven of which are of type $(1,1)$, and elementary lattice theory identifies some of the sublattices of $L$ that they cut out: two of them are lattices $A_2(-1)$, and the invariant lattices of two involutions of $G$ are isometric to $D_4(-1)$. The reflections in the roots of these lattices are the Hodge isometries of \cref{thm:B}.

The paper is organized as follows. In \cref{sec:prelim} we recall Catanese's results on canonical models, fix the notation and describe $\Aut(X)$ in terms of the normal form. In \cref{sec:bound} we prove the bound $|\Aut(X)|\leq72$, and in \cref{sec:extremal} we classify the surfaces attaining it and prove \cref{thm:A}. In \cref{sec:griffiths} we compute the Griffiths basis of $X_r$ and the intersection form, and in \cref{sec:hodge} we prove \cref{thm:B}. 

\subsection*{Acknowledgments}
We warmly thank \textsc{Roberto Villaflor Loyola} for useful discussions at various stages of this project.

\subsection*{AI statement}
The mathematical ideas, results and proofs in this paper are entirely the work of the authors, who take full responsibility for them. They arose as a natural continuation of our previous work \cite{GALMVL24}. AI tools (\textsc{Claude}) were used during the preparation of the paper as technical assistants, to write \textsc{Python} scripts verifying the lattice and linear algebra computations, and to improve the exposition and edit the text.

%----------------------------------------------------------------------
\section{Preliminaries}
\label{sec:prelim}
%----------------------------------------------------------------------

The weighted projective space we consider in this paper is
\[\PW=\Proj\CC[x_0,x_1,x_2,z_3,z_4],\]
where the variables have weights $1,2,2,3,3$, respectively. This space is well formed, and its singular locus is the disjoint union of the two lines
\[\Sigma_2=\{x_0=z_3=z_4=0\}\simeq\PP^1,\qquad \Sigma_3=\{x_0=x_1=x_2=0\}\simeq\PP^1,\]
whose points are the points with nontrivial stabilizer, $\mu_2$ and $\mu_3$ respectively, for the action of $\Gm$ on $\AA^5\smallsetminus\{0\}$ with weights $1,2,2,3,3$. We refer to \cite{Dol82,IF00} for generalities on weighted complete intersections. Over the smooth locus of $\PW$ the quotient map from the punctured affine cone is a principal $\Gm$-bundle.

\medskip

Now let $S$ be a minimal surface of general type. Its canonical model is $X=\Proj R(S,K_S)$, where $R(S,K_S)=\bigoplus_{n\geq0}\operatorname{H}^0(S,nK_S)$ is the canonical ring. The surface $X$ is a normal projective surface with at most rational double points, the natural morphism $S\to X$ is the minimal resolution, and $K_X$ is an ample Cartier divisor with $K_X^2=K_S^2$ \cite{Bom73,Rei87}. Since the minimal resolution of $X$ is unique, $\Aut(S)=\Aut(X)$, and every automorphism of $X$ acts on the canonical ring by a graded automorphism.

\begin{theorem}[Catanese]
\label{thm:catanese}
Let $S$ be a minimal surface of general type with $K_S^2=p_g(S)=1$ and let $X$ be its canonical model.
\begin{enumerate}[(1)]
\item $q(S)=0$, $S$ is simply connected and $h^0(nK_S)=\tfrac12n(n-1)+2$ for $n\geq2$.
\item The canonical ring of $S$ is generated by elements $x_0,x_1,x_2,z_3,z_4$ of degrees $1,2,2,3,3$ subject to two relations $f,g$ of degree six. Thus $X=V(f,g)\subset\PW$ is a complete intersection of two sextics, and $X$ is contained in the smooth locus of $\PW$. Conversely, a complete intersection of two sextics contained in the smooth locus of $\PW$ and with at most rational double points is the canonical model of a minimal surface of general type with $K^2=p_g=1$.
\item After a change of the generators, $f$ and $g$ have the normal form
\[f=z_3^2+x_0z_4\,a(y)+F(y),\qquad g=z_4^2+x_0z_3\,b(y)+G(y),\]
where $a,b$ are linear forms and $F,G$ are cubic forms in $y=(y_0,y_1,y_2)=(x_0^2,x_1,x_2)$. The binary cubic forms $C=F(0,y_1,y_2)$ and $D=G(0,y_1,y_2)$ are then nonzero and coprime.
\item Let $X'=V(f',g')$ be the canonical model of a second such surface, with $f',g'$ in normal form. Every isomorphism $X\to X'$ is induced by a graded automorphism of $\CC[x_0,x_1,x_2,z_3,z_4]$ carrying the ideal $(f',g')$ to the ideal $(f,g)$ and of the form
\[x_0\mapsto\lambda x_0,\qquad y_i\mapsto d_{i0}y_0+d_{i1}y_1+d_{i2}y_2\ \ (i=1,2),\qquad z_3\mapsto\mu z_3,\quad z_4\mapsto\nu z_4,\]
or of the same form with the last two maps replaced by $z_3\mapsto\mu z_4$, $z_4\mapsto\nu z_3$.
\end{enumerate}
\end{theorem}

\begin{proof}
(1) is \cite[Lemma~1.1 and Theorem~2.2]{Cat80}, (2) is \cite[Propositions~1.5, 1.6 and Theorem~1.7]{Cat80}, (3) is \cite[Proposition~1.8 and Lemma~2.1]{Cat80} and (4) is \cite[Theorem~1.7 and Proposition~1.9]{Cat80}. See also \cite{Cat79}.
\end{proof}

Two graded automorphisms of the polynomial ring inducing the same isomorphism $X\to X'$ differ by an element of the torus $\{(t,t^2,t^2,t^3,t^3)\}$, because a graded automorphism of the canonical ring which acts trivially on $X$ acts on each $\operatorname{H}^0(S,nK_S)$ by a scalar. In particular $\Aut(X)$ is the quotient by this torus of the group of graded automorphisms of the polynomial ring preserving the ideal $(f,g)$, and we identify an automorphism of $X$ with a graded automorphism of the polynomial ring inducing it.

We shall use repeatedly the following consequence of the fact that $X$ is a canonical model.

\begin{lemma}
\label{lem:jacobian}
Let $X=V(f,g)$ be as in \cref{thm:catanese} and let $\mathcal C=V(f,g)\subset\AA^5$ be the affine cone over $X$. At every point $p\neq0$ of $\mathcal C$ the Jacobian matrix of $(f,g)$ has rank at least one, and it has rank two if and only if the image of $p$ is a smooth point of $X$.
\end{lemma}

\begin{proof}
Since $X$ lies in the smooth locus of $\PW$, the map $\mathcal C\smallsetminus\{0\}\to X$ is a principal $\Gm$-bundle, so $\dim T_p\mathcal C=\dim T_PX+1$, where $P\in X$ is the image of $p\in \mathcal{C}$. As $T_p\mathcal C$ is the kernel of the Jacobian matrix, this matrix has rank two if and only if $\dim T_PX=2$, and it has rank zero if and only if $\dim T_PX=4$, which is impossible since a rational double point has embedding dimension 3.
\end{proof}

\begin{notation}
\label{not:main}
From now on $X=V(f,g)\subset\PW$ is the canonical model of a minimal surface of general type with $K^2=p_g=1$, with $f,g$ in the normal form of \cref{thm:catanese}, that is,
\[f=z_3^2+x_0z_4\,a(y)+F(y),\qquad g=z_4^2+x_0z_3\,b(y)+G(y),\]
where $y=(y_0,y_1,y_2)=(x_0^2,x_1,x_2)$ are the coordinates of $\PP(2,2,2)\simeq\PP^2$, $a,b$ are linear forms and $F,G$ are cubic forms in $y$. We write
\begin{gather*}F=c_0y_0^3+y_0^2L(y_1,y_2)+y_0Q(y_1,y_2)+C(y_1,y_2),\\ G=d_0y_0^3+y_0^2L'(y_1,y_2)+y_0Q'(y_1,y_2)+D(y_1,y_2),\end{gather*}
with $L,L'$ linear, $Q,Q'$ quadratic and $C,D$ cubic binary forms, and
\[a=a_0y_0+a_1y_1+a_2y_2,\qquad b=b_0y_0+b_1y_1+b_2y_2.\]
For a form $H$ in $y$ and $A\in\GL_2$ we write $H\circ A$ for the form $y\mapsto H(y_0,A(y_1,y_2))$. The pencil spanned by the parts of $f$ and $g$ which are quadratic in $z_3,z_4$ does not depend on the choice of $f,g$. In the normal form its two degenerate members are $z_3^2$ and $z_4^2$.
\end{notation}

\medskip

We now describe the automorphism group of $X$ in terms of the coefficients of the normal form in \cref{not:main}.

\begin{proposition}
\label{prop:autX}
Let $X$ be as in \cref{not:main}. After a further change of coordinates preserving the normal form, $\Aut(X)$ consists exactly of the automorphisms
\[x_0\mapsto x_0,\qquad (y_1,y_2)\mapsto A(y_1,y_2),\qquad z_3\mapsto \mu z_3,\qquad z_4\mapsto\nu z_4\]
with $A\in\GL_2$ and $\mu,\nu\in\Gm$ verifying
\begin{equation}
\label{eq:first}
F\circ A=\mu^2F,\qquad G\circ A=\nu^2G,\qquad \nu\,(a\circ A)=\mu^2a,\qquad \mu\,(b\circ A)=\nu^2b,
\end{equation}
together with the automorphisms obtained from these by replacing the last two maps by $z_3\mapsto\mu z_4$ and $z_4\mapsto\nu z_3$, in which case the conditions are
\begin{equation}
\label{eq:second}
F\circ A=\mu^2G,\qquad G\circ A=\nu^2F,\qquad \nu\,(a\circ A)=\mu^2b,\qquad \mu\,(b\circ A)=\nu^2a.
\end{equation}
\end{proposition}

\begin{definition}
\label{def:kinds}
We call the automorphisms of the first list in \cref{prop:autX} \emph{of the first kind} and the others \emph{of the second kind}. The elements of the first kind are the automorphisms preserving each of the two degenerate members $z_3^2$, $z_4^2$ of the pencil of \cref{not:main}, so they form a subgroup $\Aut_1(X)\subseteq\Aut(X)$ of index at most two which does not depend on the choice of the normal form.
\end{definition}

\begin{proof}[Proof of \cref{prop:autX}]
By \cref{thm:catanese}(4), every automorphism of $X$ is induced by a graded automorphism of the form $x_0\mapsto\lambda x_0$, $y\mapsto My$, with $M\in\GL_3$ stabilizing the line $\langle y_0\rangle$, together with a diagonal or antidiagonal rescaling of $z_3,z_4$. The image of $\Aut(X)$ in $\GL_3$ lies in the parabolic subgroup $P$ of matrices stabilizing $\langle y_0\rangle$, and $P=(\Gm\times\GL_2)\ltimes U_1$, where $U_1$ is the unipotent radical, consisting of the coordinate changes $y_0\mapsto y_0$, $y_i\mapsto y_i+c_iy_0$. The group of graded automorphisms of the polynomial ring preserving $(f,g)$ is an extension of the finite group $\Aut(X)$ by the torus $\{(t,t^2,t^2,t^3,t^3)\}$, so it is reductive, and its image in $P$ is conjugate to a subgroup of the Levi factor $\Gm\times\GL_2$ by an element $u\in U_1$ (see, e.g., \cite{Mos56}). The element $u$ is induced by the coordinate change $x_i\mapsto x_i+c_ix_0^2$, which preserves the normal form: it replaces $a,b,F,G$ by $a\circ u,b\circ u,F\circ u,G\circ u$ and leaves the terms $z_3^2$, $z_4^2$, $x_0z_4$ and $x_0z_3$ unchanged. After this coordinate change every element of $\Aut(X)$ acts on $y$ by a matrix $\diag(\lambda^2,A)$ with $A\in\GL_2$. The torus acts on the parameters by $\lambda\mapsto t\lambda$, $A\mapsto t^2A$, $\mu\mapsto t^3\mu$, $\nu\mapsto t^3\nu$, so each automorphism has a unique representative with $\lambda=1$.

For an automorphism $\varphi$ of the first kind with $\lambda=1$ a direct computation gives 
\[\varphi^*f=\mu^2z_3^2+\nu x_0z_4\,(a\circ A)+F\circ A,\qquad \varphi^*g=\nu^2z_4^2+\mu x_0z_3\,(b\circ A)+G\circ A,\]
and $\varphi$ preserves the ideal $(f,g)$ if and only if $\varphi^*f=\mu^2f$ and $\varphi^*g=\nu^2g$, which is \eqref{eq:first}. Indeed, since $f$ and $g$ span the degree six part of the ideal, $\varphi$ preserves $(f,g)$ if and only if $\varphi^*f$ and $\varphi^*g$ are linear combinations of $f$ and $g$, and comparing the coefficients of $z_3^2$ and $z_4^2$ shows that these combinations are $\mu^2f$ and $\nu^2g$. The same argument applies to the second kind. For the second kind one gets $\varphi^*f=\mu^2z_4^2+\nu x_0z_3(a\circ A)+F\circ A$ and $\varphi^*g=\nu^2z_3^2+\mu x_0z_4(b\circ A)+G\circ A$, so the condition is $\varphi^*f=\mu^2g$, $\varphi^*g=\nu^2f$, which is \eqref{eq:second}.
\end{proof}

Sending an automorphism to its matrix $A$ defines a homomorphism $\Aut(X)\to\GL_2$. We call $B$ its image and $V$ its kernel. Composing with $\GL_2\to\PGL_2$ we obtain a homomorphism $\Aut(X)\to\PGL_2$ with image $T$, so that we have the exact sequences
\begin{equation}
\label{eq:seq}
1\to V\to \Aut(X)\to B\to1,\qquad 1\to B\cap\Gm\to B\to T\to 1,
\end{equation}
where $\Gm\subset\GL_2$ denotes the scalar matrices, and $|\Aut(X)|=|V|\cdot|B\cap\Gm|\cdot|T|$. The group $T$ acts on the line $\PP^1$ with homogeneous coordinates $(y_1:y_2)$, and we write $T_1\subset T$ for the image of the subgroup $\Aut_1(X)$ of elements of the first kind.

In view of \eqref{eq:seq}, we bound $|\Aut(X)|$ by bounding separately the orders of $V$, $B\cap\Gm$ and $T$. The following lemma computes $V$. It also shows that the scalar matrices in $B$ come only from elements of the first kind, a fact used in \cref{lem:scalars} to bound $|B\cap\Gm|$.

\begin{lemma}
\label{lem:V}
No element of the second kind has $A$ scalar. The group $V$ consists of the elements of the first kind with $A=\Id$, and
\[V=\begin{cases}1&\text{if }a\neq0\text{ and }b\neq0,\\ \ZZ/2\ZZ&\text{if exactly one of }a,b\text{ is zero},\\ (\ZZ/2\ZZ)^2&\text{if }a=b=0.\end{cases}\]
\end{lemma}

\begin{proof}
An element of the second kind with $A=t\Id$ satisfies $F\circ A=\mu^2G$, hence $t^3C=\mu^2D$, contradicting the fact that $C$ and $D$ are nonzero and coprime. So $V$ consists of elements of the first kind with $A=\Id$, for which \eqref{eq:first} reads $F=\mu^2F$, $G=\nu^2G$, $\nu a=\mu^2a$ and $\mu b=\nu^2b$. Since $F,G\neq0$ we get $\mu^2=\nu^2=1$. If $a\neq0$ the third condition gives $\nu=1$, and if $b\neq0$ the fourth gives $\mu=1$. Conversely $(\mu,\nu)=(-1,1)$ is admissible if and only if $b=0$, and $(\mu,\nu)=(1,-1)$ if and only if $a=0$.
\end{proof}

%----------------------------------------------------------------------
\section{Bounds on the order of the automorphism group}
\label{sec:bound}
%----------------------------------------------------------------------

We keep the notation of the previous sections. By \cref{lem:V}, the group $V$ has order at most four, with equality if and only if $a=b=0$. In this section we bound the two remaining factors of $|\Aut(X)|=|V|\cdot|B\cap\Gm|\cdot|T|$. We first show that $|B\cap\Gm|\leq3$ in \cref{lem:scalars}, and then we bound $|T|$, treating separately the case in which $C$ and $D$ are both cubes of linear forms.

The point $(1:0:0:0:0)$ of $\PW$ lies on $X$ exactly when $c_0=d_0=0$. The following lemma, which will be used repeatedly in this section, is the condition that $X$ has at most a rational double point there.

\begin{lemma}
\label{lem:pointQ}
Write $L=L^{(1)}y_1+L^{(2)}y_2$ and $L'=L'^{(1)}y_1+L'^{(2)}y_2$. If $c_0=d_0=0$, then $(L,L',a_0,b_0)\neq(0,0,0,0)$.
\end{lemma}

\begin{proof}
If $c_0=d_0=0$ then $p=(1,0,0,0,0)$ lies on the affine cone, and the gradients of $f$ and $g$ at $p$ are $(0,L^{(1)},L^{(2)},0,a_0)$ and $(0,L'^{(1)},L'^{(2)},b_0,0)$. By \cref{lem:jacobian} they do not both vanish.
\end{proof}

We now bound the second factor $|B\cap\Gm|$. When the bound is attained, the lemma also shows that the cubic forms $F$ and $G$ reduce to $c_0y_0^3+C$ and $d_0y_0^3+D$, which will be the starting point of the classification in \cref{sec:extremal}.

\begin{lemma}
\label{lem:scalars}
The group $B\cap\Gm$ has order at most $3$. If $|B\cap\Gm|=3$ then
\[L=L'=Q=Q'=0,\qquad a=a_0y_0,\qquad b=b_0y_0,\]
so that $F=c_0y_0^3+C$ and $G=d_0y_0^3+D$.
\end{lemma}

\begin{proof}
By \cref{lem:V} an element with $A=t\Id \in B\cap \Gm$ is of the first kind, so \eqref{eq:first} gives
\begin{gather*}c_0=\mu^2c_0,\quad tL=\mu^2L,\quad t^2Q=\mu^2Q,\quad t^3C=\mu^2C,\quad \nu a_0=\mu^2a_0,\quad\nu ta_i=\mu^2a_i\ (i=1,2),\\
d_0=\nu^2d_0,\quad tL'=\nu^2L',\quad t^2Q'=\nu^2Q',\quad t^3D=\nu^2D,\quad \mu b_0=\nu^2b_0,\quad\mu tb_i=\nu^2b_i\ (i=1,2).\end{gather*}
Since $C\neq0$ and $D\neq0$ we get $\mu^2=\nu^2=t^3$. If $c_0\neq0$ or $d_0\neq0$, then $t^3=1$. If $c_0=d_0=0$, then by \cref{lem:pointQ} either $L\neq0$ or $L'\neq0$, in which case $t=t^3$ and $t^2=1$, or $L=L'=0$ and $a_0\neq0$ or $b_0\neq0$. If $a_0\neq0$ then $\nu=\mu^2=t^3$ and $\nu^2=t^3$ give $t^3=1$, and the same holds if $b_0\neq0$. In all cases $t^2=1$ or $t^3=1$, so $|B\cap\Gm|\leq3$.

Suppose $|B\cap\Gm|=3$ and let $t$ be a primitive cube root of unity in it. Then $\mu^2=\nu^2=1$, and from $tL=L$, $t^2Q=Q$, $\nu ta_i=a_i$ and $\mu tb_i=b_i$ we get $L=Q=0$, $a_1=a_2=0$ and $b_1=b_2=0$, and symmetrically $L'=Q'=0$.
\end{proof}

\begin{definition}
\label{def:degenerate}
We say that $X$ is of \emph{degenerate type} if both $C$ and $D$ are cubes of linear forms.
\end{definition}

We now turn to the third factor $|T|$. The group $T$ acts on $\PP^1$ preserving the union of the sets of roots of $C$ and $D$, which consists of at most six points. Outside the degenerate type these points are enough to bound $|T|$, as shown in the following lemma, which also describes the configuration of the roots when the bound on $|T_1|$ is attained. In the degenerate type the roots of $C$ and $D$ are only two points and give no bound. This case is treated at the end of this section by means of a lattice of characters, and it is the source of the automorphisms of large prime order of \cref{ex:primes}. Throughout, $\zeta_3$ denotes a primitive cube root of unity.

\begin{lemma}
\label{lem:Tbound}
Assume that $X$ is not of degenerate type. Then $T_1$ has index at most two in $T$ and $|T_1|\leq6$, with $|T_1|\leq3$ if $T_1$ is cyclic. If $|T_1|=6$ then $C$ and $D$ have three distinct roots each and, after a projective change of the coordinate $\lambda=y_1/y_2$ on $\PP^1$, their sets of roots are the two orbits of size three of the dihedral group $\Delta$ of order six generated by $\lambda\mapsto\zeta_3\lambda$ and $\lambda\mapsto\lambda^{-1}$.
\end{lemma}

\begin{proof}
An element of the first kind satisfies $C\circ A=\mu^2C$ and $D\circ A=\nu^2D$, so its image in $\PGL_2$ preserves the set of roots of $C$ and the set of roots of $D$, which are disjoint because $\gcd(C,D)=1$. An element of the second kind exchanges the two sets, so $T_1$ has index at most two in $T$.

If $C$ has three distinct roots, then $T_1$ acts faithfully on them, so $T_1$ embeds into the symmetric group on three letters $\mathfrak{S}_3$. In particular $|T_1|\leq6$, and $|T_1|\leq3$ if $T_1$ is cyclic. If $C=\ell^2m$ with $\ell,m$ independent, then $T_1$ fixes the two points cut out by $\ell$ and $m$, hence is cyclic and acts freely on the complement of these two points. The roots of $D$ avoid them and are at most three, so $|T_1|\leq3$. If $C=\ell^3$, then $T_1$ fixes the point cut out by $\ell$ and is cyclic. Since $X$ is not of degenerate type, either $D=m^2n$ with $m,n$ independent, in which case $T_1$ fixes three points and is trivial, or $D$ has three distinct roots, which $T_1$ permutes in orbits that are free away from its second fixed point, so $|T_1|\leq3$. The case in which $D$ has fewer than three distinct roots and $C$ has three is symmetric.

Finally, $|T_1|=6$ forces $C$ and $D$ to have three distinct roots and $T_1$ to be the full symmetric group on the roots of $C$. A subgroup of $\PGL_2$ isomorphic to the symmetric group on three letters is conjugate to $\Delta$, whose orbits on $\PP^1$ are $\{0,\infty\}$, the two orbits $\{1,\zeta_3,\zeta_3^2\}$ and $\{-1,-\zeta_3,-\zeta_3^2\}$, and orbits of size six. The roots of $C$ and of $D$ form two disjoint $\Delta$-stable sets of size three, so they are these two orbits.
\end{proof}

By \cref{lem:Tbound} we have $|T|\leq12$ outside the degenerate type. This bound can be improved when $B\cap\Gm$ is nontrivial, and the next two lemmas treat the cases $|B\cap\Gm|=3$ and $|B\cap\Gm|=2$. In the first one we assume moreover that $a=b=0$, which by \cref{lem:V} is the case in which $|V|$ is maximal.

\begin{lemma}
\label{lem:case3}
Assume that $X$ is not of degenerate type, that $|B\cap\Gm|=3$ and that $a=b=0$. Then $c_0$ and $d_0$ are not both zero and $|T|\leq6$.
\end{lemma}

\begin{proof}
By \cref{lem:scalars} we have $F=c_0y_0^3+C$, $G=d_0y_0^3+D$ and $L=L'=0$, and $a=b=0$ gives $a_0=b_0=0$. So $c_0$ and $d_0$ are not both zero by \cref{lem:pointQ}. By \cref{lem:Tbound} we may assume $|T_1|=6$, and we must show that $T=T_1$. By \cref{lem:Tbound} again we may take $C=y_1^3+y_2^3$ and $D=y_1^3-y_2^3$, up to scaling and up to exchanging the names of $C$ and $D$.

Suppose that some $\varphi\in\Aut(X)$ is of the second kind. Its conditions include $c_0=\mu^2d_0$ and $d_0=\nu^2c_0$, so $c_0$ and $d_0$ are both nonzero, and then every element of the first kind has $\mu^2=\nu^2=1$, that is, $C\circ A=C$ and $D\circ A=D$. A diagonal $A=\diag(t_1,t_2)$ then satisfies $t_1^3=t_2^3=1$ and induces a rotation $\lambda\mapsto(t_1/t_2)\lambda$, so a reflection in $T_1$ must be induced by an antidiagonal $A$, say with entries $p,q$. For such an $A$ one has $C\circ A=q^3y_2^3+p^3y_1^3$ and $D\circ A=q^3y_2^3-p^3y_1^3$, so $C\circ A=C$ forces $p^3=q^3=1$ and then $D\circ A=-D$, a contradiction. Hence $T_1$ contains no reflection and $|T_1|\leq3$, which is absurd.
\end{proof}

\begin{lemma}
\label{lem:case2}
Assume that $X$ is not of degenerate type and that $|B\cap\Gm|=2$. Then $a=b=0$, $(L,L')\neq(0,0)$ and $|T|\leq6$.
\end{lemma}

\begin{proof}
Let $t=-1$ be the nontrivial element of $B\cap\Gm$, with associated $\mu,\nu$, so that $\mu^2=\nu^2=t^3=-1$. From the list of conditions in the proof of \cref{lem:scalars} we get $c_0=-c_0$ and $Q=-Q$, so $c_0=0$ and $Q=0$, and symmetrically $d_0=Q'=0$. The conditions on $a$ read $\nu a_0=-a_0$ and $-\nu a_i=-a_i$ for $i=1,2$. Since $\nu^2=-1$ we have $\nu\neq\pm1$, so $a=0$, and symmetrically $b=0$. Now \cref{lem:pointQ} gives $(L,L')\neq(0,0)$.

Comparing in $F\circ A=\mu^2F$ the parts of degree two in $y_0$ gives $L\circ A=\mu^2L$, and similarly $L'\circ A=\nu^2L'$, for every element of the first kind. Hence $T_1$ fixes the point of $\PP^1$ cut out by whichever of $L,L'$ is nonzero, so it is cyclic and $|T_1|\leq3$ by \cref{lem:Tbound}. Therefore $|T|\leq6$.
\end{proof}

\medskip

Having bounded $|T|$ for the surfaces which are not of degenerate type, we now concentrate on the degenerate type. Assume then that $C=\ell^3$ and $D=m^3$ with $\ell,m$ independent linear forms, and normalize the coordinates so that $C=y_1^3$ and $D=y_2^3$. Every element of the first kind preserves the two points of $\PP^1$ cut out by $y_1$ and $y_2$, so its matrix $A=\diag(s,u)$ is diagonal, and an element of the second kind exchanges them, so its matrix is antidiagonal. The conditions of \cref{prop:autX} can therefore be compared monomial by monomial. We write $B_1$ for the image of $\Aut_1(X)$ under $(A,\mu,\nu)\mapsto(s,u)$, whose kernel is $V$ by  \cref{lem:V}, so that $|\Aut_1(X)|=|V|\cdot|B_1|$.

\begin{lemma}
\label{lem:relations}
With the above normalization, $B_1$ is contained in the subgroup of $\Gm\times\Gm$ annihilated by the following characters, one for each monomial occurring in the equations:
\begin{enumerate}[(i)]
\item $s^{j-3}u^{k}$ for each monomial $y_0^iy_1^jy_2^k$ of $F$ with $i\geq1$;
\item $s^{j}u^{k-3}$ for each monomial $y_0^iy_1^jy_2^k$ of $G$ with $i\geq1$;
\item $s^{6-2j}u^{-2k-3}$ for each monomial $y_0^{1-j-k}y_1^jy_2^k$ of $a$;
\item $s^{-2j-3}u^{6-2k}$ for each monomial $y_0^{1-j-k}y_1^jy_2^k$ of $b$.
\end{enumerate}
Let $\Gamma\subset\ZZ^2$ be the lattice generated by the corresponding exponent vectors. Then $|B_1|$ divides $|pq'-qp'|$ for any two independent vectors $(p,q)$ and $(p',q')$ of $\Gamma$. The exponent vectors of type (i) satisfy $-3\leq p\leq-1$ and $0\leq q\leq2$, those of type (ii) satisfy $0\leq p\leq2$ and $-3\leq q\leq-1$, those of type (iii) are $(6,-3)$, $(4,-3)$ and $(6,-5)$, and those of type (iv) are $(-3,6)$, $(-5,6)$ and $(-3,4)$.
\end{lemma}

\begin{proof}
The monomial $y_1^3$ of $F$ gives $s^3=\mu^2$ and the monomial $y_2^3$ of $G$ gives $u^3=\nu^2$. A monomial $y_0^iy_1^jy_2^k$ of $F$ gives $s^ju^k=\mu^2=s^3$, which is (i), and similarly for $G$. A monomial $y_0^{1-j-k}y_1^jy_2^k$ of $a$ gives $\nu s^ju^k=\mu^2=s^3$, so $\nu=s^{3-j}u^{-k}$. Squaring and using $\nu^2=u^3$ gives (iii), and (iv) is obtained in the same way from $b$. The group $B_1$ is contained in the group of characters of $\ZZ^2/\Gamma$, whose order is the index of $\Gamma$, and that index divides the index of any sublattice of finite index, in particular the absolute value of the determinant of two independent elements of $\Gamma$. The ranges of the exponents follow from $j+k\leq2$ in (i) and (ii) and from $j+k\leq1$ in (iii) and (iv).
\end{proof}

To apply \cref{lem:relations} we need enough monomials in the equations to produce two independent vectors in $\Gamma$. When $a=0$ or $b=0$, the characters of types (iii) or (iv) are not available, and the following lemma guarantees that characters of types (i) and (ii) occur instead.

\begin{lemma}
\label{lem:curve}
With the above normalization, if $a=0$ then $F$ has a monomial divisible by $y_0$, and if $b=0$ then $G$ has a monomial divisible by $y_0$.
\end{lemma}

\begin{proof}
If $a=0$ and $F=y_1^3$, then $f=z_3^2+x_1^3$, whose gradient vanishes along the plane $\Pi=\{x_1=z_3=0\}\subset\AA^5$. Since $f$ vanishes on $\Pi$ and $g$ restricted to $\Pi$ is not identically zero, $\mathcal C\cap\Pi$ is a surface contained in the affine cone, along which the Jacobian matrix has rank at most one. By \cref{lem:jacobian} its image in $X$, which is a curve, consists of singular points, contradicting the fact that $X$ has isolated singularities. The second assertion is symmetric.
\end{proof}

The following lemma excludes the only configuration of the degenerate type which would give an automorphism group of order $72$. It will be used in the proof of \cref{prop:degenerate}, in the case $|B_1|=9$ with elements of the second kind.

\begin{lemma}
\label{lem:X0}
The complete intersection in $\PW$ defined by
\[z_3^2+x_0^6+x_1^3=0,\qquad z_4^2+x_0^6+x_2^3=0\]
has a simple elliptic singularity at the point $P_0=(0:0:1:0:\sqrt{-1})$. In particular, it is not the canonical model of a surface of general type.
\end{lemma}

\begin{proof}
The point $P_0$ lies on this surface and in the smooth locus of $\PW$. We consider its lift $p=(0,0,1,0,\sqrt{-1})\in\AA^5$, and we write $f=z_3^2+x_0^6+x_1^3$ and $g=z_4^2+x_0^6+x_2^3$. The gradient of $f$ vanishes at $p$, while $\partial g/\partial x_2$ does not. Hence $\{g=0\}$ is smooth at $p$, with local coordinates $x_0$, $x_1$, $z_3$, $z_4$, and the cone over the surface is cut out on it by $f$, which does not involve $z_4$. The orbit of the $\Gm$-action at $p$ is the $z_4$-direction, so a transverse slice of the cone at $p$, which is a local model of the surface at $P_0$, is the surface singularity $z_3^2+x_1^3+x_0^6=0$. This is the simple elliptic singularity $\widetilde E_8$, which is not a rational double point.
\end{proof}

We can now bound the order of the automorphism group in the degenerate type. The proof combines \cref{lem:V} with the bounds on $|B_1|$ given by \cref{lem:relations}, distinguishing whether $a$ and $b$ vanish.

\begin{proposition}
\label{prop:degenerate}
If $X$ is of degenerate type then $|\Aut(X)|\leq64$.
\end{proposition}

\begin{proof}
We use \cref{lem:relations} throughout, and $|\Aut(X)|$ equals $|\Aut_1(X)|$ or $2\,|\Aut_1(X)|$, according to whether there are elements of the second kind or not. Such an element satisfies $\nu(a\circ A)=\mu^2b$ and $\mu(b\circ A)=\nu^2a$, so $a=0$ if and only if $b=0$.

Assume first $a=b=0$, so $|V|=4$ by \cref{lem:V}. By \cref{lem:curve} the lattice $\Gamma$ contains a vector $(p,q)$ of type (i) and a vector $(p',q')$ of type (ii). Since $j+k\leq2$ in (i) and (ii), we have $p+q\leq-1$ and $p'+q'\leq-1$, that is, $-p>q\geq0$ and $-q'>p'\geq0$. Hence $pq'=(-p)(-q')>qp'$, so the two vectors are independent and $1\leq pq'-qp'\leq pq'\leq9$. Thus $|B_1|\leq9$, whence $|\Aut_1(X)|\leq36$. If there is no element of the second kind we are done. If there is one, its matrix is antidiagonal, so it carries the monomial $y_0^iy_1^jy_2^k$ to a multiple of $y_0^iy_2^jy_1^k$. Comparing $F\circ A=\mu^2G$ and $G\circ A=\nu^2F$ we see that the support of $G$ is obtained from that of $F$ by exchanging $y_1$ and $y_2$, so $\Gamma$ is stable under $(p,q)\mapsto(q,p)$. Suppose $|B_1|=9$. Then $9$ divides $pq'-qp'$ for every pair of vectors of types (i) and (ii) in $\Gamma$, and since $1\leq pq'-qp'\leq9$ we get $pq'-qp'=9$, hence $p=q'=-3$. Taking $(p',q')=(q,p)$, we obtain that $9$ divides $9-q^2$, so $q=0$. Thus $(-3,0)$ and $(0,-3)$ lie in $\Gamma$, and since the index of $\Gamma$ is at least $|B_1|=9$ we get $\Gamma=3\ZZ\times3\ZZ$. Every vector of type (i) then equals $(-3,0)$ and every vector of type (ii) equals $(0,-3)$, so the only monomial of $F$ divisible by $y_0$ is $y_0^3$, and likewise for $G$. Rescaling the variables we obtain the complete intersection of \cref{lem:X0}, which is excluded. Hence $|B_1|\leq8$ and $|\Aut(X)|=2|\Aut_1(X)|\leq2\cdot4\cdot8=64$.

Similarly, if we assume that exactly one of $a,b$ vanishes, say $b=0$ and $a\neq0$, so $|V|=2$ and there is no element of the second kind. By \cref{lem:curve} the lattice $\Gamma$ contains a vector $(p',q')$ of type (ii), and it contains one of the three vectors of type (iii). These are independent. A direct computation shows that the corresponding determinants are $6q'+3p'$, $4q'+3p'$ and $6q'+5p'$, which for $0\leq p'\leq2$ and $-3\leq q'\leq-1$ have absolute value at most $18$. Hence $|B_1|\leq18$ and $|\Aut(X)|=|\Aut_1(X)|\leq36$.

Finally assume $a\neq0$ and $b\neq0$, so $|V|=1$. Then $\Gamma$ contains one of the three vectors of type (iii) and one of the three of type (iv), which are independent, and the nine corresponding determinants have absolute values $27$, $21$, $15$, $15$, $9$, $7$, $21$, $11$ and $9$. Hence $|B_1|\leq27$ and $|\Aut(X)|\leq2\cdot27=54$.
\end{proof}

The following examples show that the bound of \cref{prop:degenerate} is sharp, and that there are surfaces of degenerate type with automorphisms of prime order $7$ and $13$, which cannot occur outside the degenerate type, since there $|\Aut(X)|=|V|\cdot|B\cap\Gm|\cdot|T|$ has only the prime factors $2$ and $3$ by \cref{lem:V,lem:scalars,lem:Tbound}. The singular loci are computed by means of \cref{lem:jacobian}, as the points of $X$ where the gradients of $f$ and $g$ are proportional. In each case $C$ and $D$ are coprime, so the surface lies in the smooth locus of $\PW$ and is a canonical model by \cref{thm:catanese}(2).

\begin{example}
\label{ex:primes}
(1) The complete intersection $z_3^2+x_0^4x_2+x_1^3=z_4^2+x_0^4x_1+x_2^3=0$ is of degenerate type, with $C=y_1^3$, $D=y_2^3$ and $a=b=0$. Its singular points are $(0:0:1:0:\sqrt{-1})$ and $(0:1:0:\sqrt{-1}:0)$, both of type $E_6$. The conditions \eqref{eq:first} read $u=\mu^2$, $s^3=\mu^2$, $s=\nu^2$ and $u^3=\nu^2$, whence $u=s^3$ and $s^8=1$. So $B_1$ is cyclic of order eight and $|\Aut_1(X)|=32$. The automorphism $x_1\leftrightarrow x_2$, $z_3\leftrightarrow z_4$ exchanges the two equations, so $|\Aut(X)|=64$ and the bound of \cref{prop:degenerate} is attained.

(2) The complete intersection $z_3^2+x_0^4x_2+x_1^3=z_4^2+x_0x_1z_3+x_2^3=0$ is of degenerate type, with $a=0$ and $b=y_1$. Its singular points are $(1:0:0:0:0)$, of type $A_1$, and $(0:0:1:0:\sqrt{-1})$, of type $E_6$. The conditions \eqref{eq:first} read $u=\mu^2$, $s^3=\mu^2$, $u^3=\nu^2$ and $\mu s=\nu^2$, which give $u=s^3$, $\mu=s^8$ and $s^{13}=1$. Hence the coordinate change
\[x_0\mapsto x_0,\qquad x_1\mapsto\zeta x_1,\qquad x_2\mapsto\zeta^3x_2,\qquad z_3\mapsto\zeta^8z_3,\qquad z_4\mapsto\zeta^{11}z_4,\qquad\zeta^{13}=1,\]
which multiplies the first equation by $\zeta^3$ and the second by $\zeta^9$, is an automorphism of order thirteen, and $|\Aut(X)|=26$.

(3) The complete intersection $z_3^2+x_0^4x_2+x_1^3=z_4^2+x_0^2x_1^2+x_2^3=0$ is of degenerate type, with $C=y_1^3$, $D=y_2^3$ and $a=b=0$. Its singular points are $(1:0:0:0:0)$, $(0:0:1:0:\sqrt{-1})$ and $(0:1:0:\sqrt{-1}:0)$, of types $A_5$, $E_6$ and $A_2$, respectively. The conditions \eqref{eq:first} read $u=\mu^2$, $s^3=\mu^2$, $s^2=\nu^2$ and $u^3=\nu^2$, whence $u=s^3$ and $s^7=1$, so $B_1$ is cyclic of order seven and $|\Aut_1(X)|=28$. The supports of $F$ and $G$ are not exchanged by $y_1\leftrightarrow y_2$, so there is no element of the second kind and $|\Aut(X)|=28$. A generator of $B_1$ gives the automorphism of order seven
\[x_0\mapsto x_0,\qquad x_1\mapsto\zeta x_1,\qquad x_2\mapsto\zeta^3x_2,\qquad z_3\mapsto\zeta^5z_3,\qquad z_4\mapsto\zeta z_4,\qquad \zeta^7=1.\]
\end{example}

Combining the results of this section we obtain the bound of \cref{thm:A}, together with the restrictions on the surfaces attaining it which are the starting point of \cref{sec:extremal}.

\begin{theorem}
\label{thm:bound}
Let $X$ be the canonical model of a minimal surface of general type with $K^2=p_g=1$. Then $|\Aut(X)|\leq72$, and if equality holds then $X$ is not of degenerate type, $a=b=0$, $|B\cap\Gm|=3$ and $|T|=6$, where $B$ and $T$ are the images of $\Aut(X)$ in $\GL_2$ and in $\PGL_2$, as in \eqref{eq:seq}.
\end{theorem}

\begin{proof}
If $X$ is of degenerate type this is \cref{prop:degenerate}, so assume it is not. We use $|\Aut(X)|=|V|\cdot|B\cap\Gm|\cdot|T|$ together with $|V|\leq4$ (\cref{lem:V}), $|B\cap\Gm|\leq3$ (\cref{lem:scalars}) and $|T|\leq12$ (\cref{lem:Tbound}).

Suppose $|B\cap\Gm|=3$. By \cref{lem:scalars} we have $a=a_0y_0$ and $b=b_0y_0$. If $a_0=b_0=0$ then $|V|=4$ and $|T|\leq6$ by \cref{lem:case3}, so $|\Aut(X)|\leq72$. If exactly one of $a_0,b_0$ vanishes, say $b_0=0$ and $a_0\neq0$, then $|V|=2$, and there is no element of the second kind because such an element forces $a=0$. So $T=T_1$ and $|\Aut(X)|\leq2\cdot3\cdot6=36$. If $a_0b_0\neq0$ then $|V|=1$ and $|\Aut(X)|\leq3\cdot12=36$.

Suppose $|B\cap\Gm|=2$. Then $|\Aut(X)|\leq4\cdot2\cdot6=48$ by \cref{lem:case2}. Suppose $|B\cap\Gm|=1$. Then $|\Aut(X)|\leq4\cdot1\cdot12=48$.

In all cases $|\Aut(X)|\leq72$, and equality forces the first case with $a_0=b_0=0$, that is $a=b=0$, together with $|V|=4$, $|B\cap\Gm|=3$ and $|T|=6$.
\end{proof}

%----------------------------------------------------------------------
\section{The surfaces with the largest automorphism group}
\label{sec:extremal}
%----------------------------------------------------------------------

In this section we classify the surfaces attaining the bound of \cref{thm:bound}, we compute their automorphism groups and singularities, and we determine the isomorphisms among them, which proves \cref{thm:A}. Throughout this section $X$ is a canonical model with $|\Aut(X)|=72$. By \cref{thm:bound} it is not of degenerate type, $a=b=0$, and the groups $V$, $B$ and $T$ of \eqref{eq:seq} satisfy $|V|=4$, $|B\cap\Gm|=3$ and $|T|=6$. Moreover, by \cref{lem:scalars}
\[F=c_0y_0^3+C,\qquad G=d_0y_0^3+D,\]
with $(c_0,d_0)\neq(0,0)$ by \cref{lem:case3}. Since $|T|=6$ and $T_1$ has index at most two in $T$, either $T=T_1$ has order six, or $T_1$ has order three and there are elements of the second kind. We treat the two cases separately, and we find one surface in the first case, one surface in the second case when the second kind acts on $\PP^1$ through the torus normalizing $T_1$, and a one-parameter family when it acts through a reflection. For short, we write $\sigma_1,\sigma_2$ for the automorphisms $x_1\mapsto\zeta_3x_1$ and $x_2\mapsto\zeta_3x_2$ (with $\mu=\nu=1$) and $\varepsilon_1,\varepsilon_2$ for the automorphisms $z_3\mapsto-z_3$ and $z_4\mapsto-z_4$, which generate $V$.

We begin with the case $T=T_1$, in which the automorphisms of the first kind already act on $\PP^1$ through a group of order six. It leads to a single surface, with two singular points.

\begin{proposition}
\label{prop:X1}
Suppose $T=T_1$. Then $X$ is isomorphic to
\[X_{2D_4}\colon\quad z_3^2+x_0^6+x_1^3+x_2^3=0,\qquad z_4^2+x_1^3-x_2^3=0.\]
The surface $X_{2D_4}$ is a canonical model whose singular locus consists of the two points $(1:0:0:\pm\sqrt{-1}:0)$, both of type $D_4$, and $\Aut(X_{2D_4})\simeq((\ZZ/3\ZZ)^2\rtimes\ZZ/4\ZZ)\times\ZZ/2\ZZ$, where the cyclic group of order four acts through its quotient of order two by exchanging the two factors.
\end{proposition}

\begin{proof}
By \cref{lem:Tbound}, since $|T_1|=6$, after a coordinate change and a rescaling of $f$ and $g$ we may assume that $\{C,D\}=\{y_1^3+y_2^3,\ y_1^3-y_2^3\}$. The coordinate change $x_2\mapsto-x_2$ exchanges these two forms and fixes everything else, so composing it with the exchange of $f$ and $g$ and of $z_3$ and $z_4$ we may permute $c_0$ and $d_0$ while keeping $C=y_1^3+y_2^3$ and $D=y_1^3-y_2^3$. As computed in the proof of \cref{lem:case3}, an antidiagonal $A$ inducing a reflection satisfies $p^3=q^3=\mu^2$ and $\nu^2=-\mu^2$. From $c_0=\mu^2c_0$ and $d_0=\nu^2d_0=-\mu^2d_0$ we conclude that $c_0$ and $d_0$ cannot both be nonzero, so exactly one of them vanishes, and by the previous permutation we may assume $d_0=0$ and $c_0\neq0$. Substituting $x_0\mapsto\lambda x_0$ multiplies $c_0$ by $\lambda^6$, so we may normalize $c_0=1$ and obtain $X_{2D_4}$.

The equations of $X_{2D_4}$ are in normal form with $a=b=0$, and $C$ and $D$ are coprime. The gradients are $(6x_0^5,3x_1^2,3x_2^2,2z_3,0)$ and $(0,3x_1^2,-3x_2^2,0,2z_4)$. If $z_4\neq0$, proportionality of the gradients forces the first one to vanish, hence $x_0=x_1=x_2=z_3=0$ and $g=z_4^2\neq0$. So $z_4=0$ at a point where the gradients are proportional, and $g=0$ reads $x_1^3=x_2^3$. If the second gradient is nonzero, the first must be proportional to it, which forces $x_0=z_3=0$, and then $f=0$ gives $x_1^3+x_2^3=0$, so $x_1=x_2=0$ and the point is the origin. Hence the second gradient vanishes, that is $x_1=x_2=0$, and $f=0$ gives $z_3^2=-x_0^6$ with $x_0\neq0$. Normalizing $x_0=1$ we obtain the two points of the statement, which are distinct in $\PW$. Near $(1:0:0:\sqrt{-1}:0)$ we work in the chart $x_0=1$. Since $\partial f/\partial z_3\neq0$ there, $\{f=0\}$ is smooth with local coordinates $x_1,x_2,z_4$, and $X_{2D_4}$ is cut out on it by $g=z_4^2+x_1^3-x_2^3$, where $x_1^3-x_2^3$ is a product of three pairwise independent linear forms. So the singularity is of type $D_4$. Hence $X_{2D_4}$ has only rational double points, and it is a canonical model by \cref{thm:catanese}(2).

We compute $\Aut(X_{2D_4})$. Since $a=b=0$ we have $|V|=4$. Since $c_0=1$, every element of the first kind satisfies $\mu^2=1$ and $C\circ A=C$. Since $d_0=0$, the only condition coming from $G$ is $D\circ A=\nu^2D$. If $A=\diag(t_1,t_2)$ then $C\circ A=t_1^3y_1^3+t_2^3y_2^3$, so $t_1^3=t_2^3=1$, and then $D\circ A=D$ and $\nu^2=1$. This gives nine matrices. If $A$ is antidiagonal with entries $p,q$ then $C\circ A=q^3y_2^3+p^3y_1^3$, so $p^3=q^3=1$, and then $D\circ A=-D$ and $\nu^2=-1$. This gives nine more. There is no element of the second kind, since $F\circ A=\mu^2G$ is impossible: the left hand side has a nonzero coefficient in $y_0^3$ and the right hand side does not. Hence $|B|=18$ and $|\Aut(X_{2D_4})|=4\cdot18=72$.

Let $\tau$ be given by $x_1\leftrightarrow x_2$, $z_3\mapsto z_3$ and $z_4\mapsto\sqrt{-1}\,z_4$. Then $\sigma_1,\sigma_2$ generate a group isomorphic to $(\ZZ/3\ZZ)^2$, $\tau$ has order four with $\tau^2=\varepsilon_2$, and $\tau$ exchanges $\sigma_1$ and $\sigma_2$. The element $\varepsilon_1$ is central and does not belong to the group generated by $\sigma_1,\sigma_2,\tau$, which has order $36$. Hence $\Aut(X_{2D_4})=((\ZZ/3\ZZ)^2\rtimes\langle\tau\rangle)\times\langle\varepsilon_1\rangle$.
\end{proof}

We now turn to the case in which $T_1$ has order three and there are elements of the second kind. The following lemma gives a common normal form for $C$ and $D$, and shows that the matrix of an element of the second kind is either diagonal or antidiagonal. These two possibilities are treated in \cref{prop:X2} and \cref{prop:family}, respectively.

\begin{lemma}
\label{lem:caseII}
Suppose $|T_1|=3$ and that there is an element of the second kind. Then $C$ and $D$ have three distinct roots each, and after a coordinate change and a rescaling of $f$ and $g$ we may write
\[C=y_1^3-\alpha y_2^3,\qquad D=y_1^3-\beta y_2^3,\qquad \alpha,\beta\in\Gm,\quad \alpha\neq\beta,\]
with $T_1$ generated by $\lambda\mapsto\zeta_3\lambda$, where $\lambda=y_1/y_2$. Moreover $c_0d_0\neq0$, and the matrix of an element of the second kind is diagonal or antidiagonal.
\end{lemma}

\begin{proof}
Normalize $T_1=\langle\lambda\mapsto\zeta_3\lambda\rangle$, whose fixed points are $0$ and $\infty$. The set of roots of $C$ is a $T_1$-stable multiset of size three, so it is either a free orbit, consisting of three distinct points, or supported on $\{0,\infty\}$, in which case $C$ has a root of multiplicity at least two. An element of the second kind maps the multiset of roots of $C$ onto that of $D$, so the two multisets have the same multiplicities. If both were supported on $\{0,\infty\}$ then, being coprime, $C$ and $D$ would be $y_1^3$ and $y_2^3$ up to scalars and $X$ would be of degenerate type. Hence both are free orbits and $C,D$ have the asserted shape after rescaling $y_1$ and the equations.

An element of the second kind satisfies $c_0=\mu^2d_0$ and $d_0=\nu^2c_0$, so $c_0$ and $d_0$ are both zero or both nonzero, and the first case is excluded by \cref{lem:case3}. Its image in $\PGL_2$ normalizes $T_1$, and a direct computation shows that the normalizer of $\langle\lambda\mapsto\zeta_3\lambda\rangle$ in $\PGL_2$ consists of the maps $\lambda\mapsto t\lambda$ and $\lambda\mapsto t/\lambda$, which are induced by diagonal and antidiagonal matrices respectively.
\end{proof}

\begin{proposition}
\label{prop:X2}
Suppose $|T_1|=3$ and that some element of the second kind has diagonal matrix. Then $X$ is isomorphic to
\[X_{3A_2}\colon\quad z_3^2+x_0^6+x_1^3-x_2^3=0,\qquad z_4^2+x_0^6+x_1^3+x_2^3=0.\]
The surface $X_{3A_2}$ is a canonical model whose singular locus consists of the three points $(1:\eta:0:0:0)$ with $\eta^3=-1$, all of type $A_2$, and $\Aut(X_{3A_2})\simeq(\ZZ/3\ZZ)^2\times \mathbf{D}_4$.
\end{proposition}

\begin{proof}
Keep the notation of \cref{lem:caseII} and let $A=\diag(t_1,t_2)$ be the matrix of an element of the second kind. Comparing the coefficients of $y_1^3$ and of $y_2^3$ in $F\circ A=\mu^2G$ and in $G\circ A=\nu^2F$ gives
\[t_1^3=\mu^2,\qquad \alpha t_2^3=\mu^2\beta,\qquad t_1^3=\nu^2,\qquad \beta t_2^3=\nu^2\alpha.\]
Hence $\mu^2=\nu^2$ and $\alpha^2=\beta^2$, that is $\alpha=-\beta$. Rescaling $y_2$ we may assume $\beta=1$ and $\alpha=-1$. From $c_0=\mu^2d_0$ and $d_0=\nu^2c_0=\mu^2c_0$ we get $\mu^4=1$, hence $c_0=\pm d_0$, and rescaling $y_0$ we may take $d_0=1$ and $c_0=\pm1$. If $c_0=-1$, the coordinate change $y_1\leftrightarrow y_2$ fixes $C$ and replaces $D$ by $-D$. Multiplying $g$ by $-1$ and replacing $z_4$ by $\sqrt{-1}\,z_4$ replaces $(c_0,d_0)=(-1,1)$ by $(-1,-1)$, and rescaling $y_0$ by a cube root of $-1$ then gives $(1,1)$. So we may assume $c_0=d_0=1$, and after exchanging the names of $C$ and $D$ we obtain $X_{3A_2}$.

The gradients of $X_{3A_2}$ are $(6x_0^5,3x_1^2,-3x_2^2,2z_3,0)$ and $(6x_0^5,3x_1^2,3x_2^2,0,2z_4)$. If $z_3z_4\neq0$ they are independent. If $z_3=0$ and $z_4\neq0$, proportionality forces the first gradient to vanish, hence $x_0=x_1=x_2=0$ and $g=z_4^2\neq0$. The case $z_3\neq0$, $z_4=0$ is symmetric. So $z_3=z_4=0$, and subtracting the equations gives $x_2=0$ and $x_1^3=-x_0^6$. There the two gradients coincide, so these points are singular, and normalizing $x_0=1$ we obtain the three points of the statement. Near such a point $\partial f/\partial x_1\neq0$, so $\{f=0\}$ is smooth with local coordinates $x_2,z_3,z_4$, and $X_{3A_2}$ is cut out on it by $g-f=z_4^2-z_3^2+2x_2^3$, an $A_2$ singularity. Hence $X_{3A_2}$ is a canonical model.

For the automorphism group, $|V|=4$ since $a=b=0$. Since $c_0=d_0=1$, every element of the first kind satisfies $\mu^2=\nu^2=1$, hence $C\circ A=C$ and $D\circ A=D$. For $A=\diag(t_1,t_2)$ this gives $t_1^3=t_2^3=1$, nine matrices. For $A$ antidiagonal with entries $p,q$ we get $C\circ A=q^3y_2^3-p^3y_1^3=C$, so $p^3=q^3=-1$, and then $D\circ A=-D$, so there is none. For an element of the second kind we need $\mu^2=\nu^2=1$, $C\circ A=D$ and $D\circ A=C$. For $A=\diag(t_1,t_2)$ this reads $t_1^3=1$ and $t_2^3=-1$, giving nine matrices, and for $A$ antidiagonal there is none. Hence $|B|=18$ and $|\Aut(X_{3A_2})|=72$. Let $\varphi$ be the automorphism $x_2\mapsto-x_2$, $z_3\leftrightarrow z_4$, which exchanges $f$ and $g$ and has order two. Then $\varphi$ commutes with $\sigma_1$ and $\sigma_2$ and exchanges $\varepsilon_1$ and $\varepsilon_2$, so $\langle V,\varphi\rangle$ is dihedral of order eight and $\Aut(X_{3A_2})=\langle\sigma_1,\sigma_2\rangle\times\langle V,\varphi\rangle$.
\end{proof}

\begin{proposition}
\label{prop:family}
Suppose $|T_1|=3$ and that every element of the second kind has antidiagonal matrix. Then $X$ is isomorphic to a member of the family
\[X_r\colon\quad z_3^2+rx_0^6+x_1^3-r^2x_2^3=0,\qquad z_4^2+x_0^6+x_1^3-x_2^3=0,\qquad r\in\Gm\smallsetminus\{\pm1\}.\]
Every $X_r$ is smooth, and $\Aut(X_r)\simeq(\ZZ/6\ZZ\times\ZZ/6\ZZ)\rtimes\ZZ/2\ZZ$ has order $72$, the group of order two exchanging the two factors.
\end{proposition}

\begin{proof}
Keep the notation of \cref{lem:caseII} and let the element of the second kind be given by the antidiagonal transformation $y_1\mapsto qy_2$, $y_2\mapsto py_1$. From $F\circ A=\mu^2G$ and $G\circ A=\nu^2F$ we get
\[c_0=\mu^2d_0,\qquad -\alpha p^3=\mu^2,\qquad q^3=-\mu^2\beta,\qquad d_0=\nu^2c_0,\qquad -\beta p^3=\nu^2,\qquad q^3=-\nu^2\alpha.\]
The two expressions for $q^3$ give $\mu^2\beta=\nu^2\alpha$, and the first and fourth equations give $\mu^2\nu^2=1$, so $\mu^4=\alpha/\beta$ and $c_0/d_0=\mu^2$. Rescaling $y_2$ we may take $\beta=1$, and rescaling $y_0$ we may take $d_0=1$. Setting $r=\mu^2$ we obtain $\alpha=r^2$ and $c_0=r$, which is $X_r$. Here $r\neq0$, and $\alpha\neq\beta$ gives $r\neq\pm1$. Conversely, given $r\in\Gm\smallsetminus\{\pm1\}$, choose $\mu$ with $\mu^2=r$, put $\nu=\mu^{-1}$, and choose $p,q$ with $p^3=-r^{-1}$, $q^3=-r$ and $pq=1$. The resulting graded automorphism is an element of the second kind of $X_r$.

We check that $X_r$ is smooth. The gradients of $f$ and $g$ are
\[\nabla f=(6rx_0^5,3x_1^2,-3r^2x_2^2,2z_3,0),\qquad \nabla g=(6x_0^5,3x_1^2,-3x_2^2,0,2z_4).\]
 If $z_4\neq0$ then proportionality forces the first gradient to vanish, hence $x_0=x_1=x_2=z_3=0$ and $g=z_4^2\neq0$. The case $z_3\neq0$ is symmetric. So $z_3=z_4=0$, and writing $\xi=x_0^6$, $\eta=x_1^3$, $\zeta=x_2^3$ the equations read $r\xi+\eta-r^2\zeta=0$ and $\xi+\eta-\zeta=0$, whence $\xi=(r+1)\zeta$ and $\eta=-r\zeta$, as $r\neq1$. If the gradients are proportional and $x_1\neq0$, the second coordinate gives the factor of proportionality $1$, and the first then gives $6rx_0^5=6x_0^5$, so $x_0=0$, hence $\zeta=0$ because $r\neq-1$, hence $\eta=0$, a contradiction. So $x_1=0$, hence $\zeta=0$ because $r\neq0$, hence $\xi=0$, and the point is the origin. Therefore $X_r$ is smooth, and it lies in the smooth locus of $\PW$ since $C$ and $D$ are coprime for $r^2\neq1$. By \cref{thm:catanese}(2) it is a canonical model.

We compute $\Aut(X_r)$. Since $c_0=r$ and $d_0=1$ are nonzero, every element of the first kind has $\mu^2=\nu^2=1$ and $C\circ A=C$, $D\circ A=D$. For $A=\diag(s,u)$ this reads $s^3=u^3=1$, giving nine matrices and, with the four choices of $(\mu,\nu)$, thirty-six elements of the first kind. For $A$ antidiagonal with entries $p,q$ we would get $-\alpha p^3=1$ and $-\beta p^3=1$, hence $\alpha=\beta$, which is excluded. By the first paragraph there are elements of the second kind, and the same count gives thirty-six of them. Hence $|\Aut(X_r)|=72$. Let $\varphi$ be the element of the second kind constructed above, with $pq=1$ and $\nu=\mu^{-1}$. Then $\varphi^2$ acts by $y_i\mapsto pqy_i$, $z_3\mapsto\mu\nu z_3$ and $z_4\mapsto\mu\nu z_4$, so $\varphi^2=1$, and conjugation by $\varphi$ exchanges $\sigma_1$ with $\sigma_2$ and $\varepsilon_1$ with $\varepsilon_2$. Hence $\langle\sigma_1,\varepsilon_1\rangle$ and $\langle\sigma_2,\varepsilon_2\rangle$ are two commuting cyclic groups of order six exchanged by $\varphi$, and $\Aut(X_r)$ is their product extended by $\langle\varphi\rangle$.
\end{proof}

It remains to determine which of the surfaces $X_{2D_4}$, $X_{3A_2}$ and $X_r$ are isomorphic. For this we attach to each surface of the type of \cref{lem:caseII} two quantities, and we show that an isomorphism can only change them in four explicit ways.

\begin{proposition}
\label{prop:moduli}
Let $X$ and $X'$ be two surfaces of the type treated in \cref{lem:caseII}, with data
\[F=c_0y_0^3+c\,(y_1^3-\alpha y_2^3),\qquad G=d_0y_0^3+d\,(y_1^3-\beta y_2^3)\]
and similarly for $X'$. Then the quantities
\[I_1=\frac{\alpha}{\beta},\qquad I_2=\frac{c_0d}{c\,d_0}\]
are defined and nonzero, and every isomorphism $X\to X'$ carries $(I_1,I_2)$ to one of
\[(I_1,I_2),\qquad\Big(\frac1{I_1},\frac{I_2}{I_1}\Big),\qquad\Big(\frac1{I_1},\frac1{I_2}\Big),\qquad\Big(I_1,\frac{I_1}{I_2}\Big).\]
Consequently $X_r$ and $X_{r'}$ are isomorphic if and only if $r'=r^{\pm1}$, and no $X_r$ is isomorphic to $X_{2D_4}$ or to $X_{3A_2}$.
\end{proposition}

\begin{proof}
The quantities are defined because $c_0d_0\neq0$ by \cref{lem:caseII} and $c,d,\alpha,\beta$ are nonzero. By \cref{thm:catanese}(4) an isomorphism $X\to X'$ is induced by a graded automorphism
\[x_0\mapsto\lambda x_0,\qquad y_0\mapsto\lambda^2y_0,\qquad y_i\mapsto (Ay)_i+e_iy_0\ (i=1,2),\]
with $A\in\GL_2$ and $e_1,e_2\in\CC$, together with either $z_3\mapsto\mu z_3$, $z_4\mapsto\nu z_4$ or $z_3\mapsto\mu z_4$, $z_4\mapsto\nu z_3$. It conjugates $\Aut_1(X)$ onto $\Aut_1(X')$, hence $T_1$ onto $T_1'$. After the normalization of \cref{lem:caseII} both are generated by $\lambda\mapsto\zeta_3\lambda$, so the class of $A$ in $\PGL_2$ normalizes that group and $A$ is diagonal or antidiagonal.

We show that $e_1=e_2=0$. Suppose $A=\diag(\rho_1,\rho_2)$. The antidiagonal case is identical after exchanging the two variables. The pullback of $F'$ under this graded automorphism is
\[c_0'\lambda^6y_0^3+c'\big((\rho_1y_1+e_1y_0)^3-\alpha'(\rho_2y_2+e_2y_0)^3\big),\]
whose coefficients of $y_0y_1^2$ and of $y_0y_2^2$ are $3c'\rho_1^2e_1$ and $-3c'\alpha'\rho_2^2e_2$. This pullback is proportional to $F$ or to $G$, neither of which involves these monomials, so $e_1=e_2=0$.

So the isomorphism is the composition of a coordinate change $y_0\mapsto\gamma y_0$, $y_1\mapsto\rho_1y_1$, $y_2\mapsto\rho_2y_2$ with rescalings of the two equations by $\kappa_1,\kappa_2$, possibly followed by the antidiagonal coordinate change $y_1\mapsto qy_2$, $y_2\mapsto py_1$, and possibly by the exchange of $f$ and $g$. Under the first ones,
\[c_0\mapsto c_0\gamma^3/\kappa_1,\quad c\mapsto c\rho_1^3/\kappa_1,\quad \alpha\mapsto\alpha\rho_2^3/\rho_1^3,\quad d_0\mapsto d_0\gamma^3/\kappa_2,\quad d\mapsto d\rho_1^3/\kappa_2,\quad \beta\mapsto\beta\rho_2^3/\rho_1^3,\]
so $I_1$ and $I_2$ are unchanged. Under the antidiagonal coordinate change one computes $c\mapsto-c\alpha p^3$ and $\alpha\mapsto q^3/(\alpha p^3)$, and similarly for $d,\beta$, while $c_0$ and $d_0$ are unchanged. This gives $I_1\mapsto1/I_1$ and $I_2\mapsto I_2/I_1$. Exchanging $f$ and $g$ exchanges the triples $(c_0,c,\alpha)$ and $(d_0,d,\beta)$, so $I_1\mapsto1/I_1$ and $I_2\mapsto1/I_2$. The four displayed possibilities are the four compositions.

For $X_r$ we have $(I_1,I_2)=(r^2,r)$, and the four possibilities give $(r^{\pm2},r^{\pm1})$. Hence $X_r\simeq X_{r'}$ forces $r'=r^{\pm1}$, and conversely $X_r\simeq X_{1/r}$ through the antidiagonal coordinate change. For $X_{3A_2}$ we have $(I_1,I_2)=(-1,1)$. The relation $I_1=I_2^2$ is preserved by the four maps and holds for every $X_r$, so no $X_r$ is isomorphic to $X_{3A_2}$. Finally, $\Aut_1(X_{2D_4})=\Aut(X_{2D_4})$ whereas $\Aut_1(X)$ has index two in $\Aut(X)$ for $X_{3A_2}$ and for every $X_r$. Since $\Aut_1$ does not depend on the normal form, $X_{2D_4}$ is not isomorphic to any of the others.
\end{proof}

This concludes the proof of \cref{thm:A}. Indeed, since $\Aut(S)=\Aut(X)$, the bound $|\Aut(S)|\leq72$ is \cref{thm:bound}. If equality holds, then either $T=T_1$, or $T_1$ has order three and there are elements of the second kind, whose matrices are diagonal or antidiagonal by \cref{lem:caseII}. In these three cases $X$ is isomorphic to $X_{2D_4}$, to $X_{3A_2}$ or to a member of the family $X_r$, by \cref{prop:X1,prop:X2,prop:family} respectively. Conversely, these propositions show that all these surfaces are canonical models with the singularities and the automorphism groups of order $72$ stated in \cref{thm:A}. Finally, by \cref{prop:moduli} the only isomorphisms among them are $X_r\simeq X_{r^{-1}}$.

\section{The Griffiths basis of $X_r$ and the intersection form}
\label{sec:griffiths}
%----------------------------------------------------------------------

We now describe the primitive cohomology of the surfaces $X_r$ of \cref{prop:family}. Using the Cayley trick, we construct a basis of $\operatorname{H}^2(X_r,\CC)_\prim$ consisting of eigenvectors of a subgroup $G\simeq\ZZ/6\ZZ\times\ZZ/6\ZZ$ of $\Aut(X_r)$ and formed by residues of explicit polynomials, most of them monomials. This is the analogue of the Griffiths basis of the primitive cohomology of a hypersurface. We then compute the intersection form in this basis, and we use the action of $G$ to decompose the rational primitive cohomology into pieces defined over $\QQ$, which is the starting point of \cref{sec:hodge}.

In this section and the next, $r\in\Gm\smallsetminus\{\pm1\}$ is fixed and $X=X_r$ is the smooth surface of \cref{prop:family}, with equations
\[f=z_3^2+rx_0^6+x_1^3-r^2x_2^3,\qquad g=z_4^2+x_0^6+x_1^3-x_2^3.\]
We put $m_0=x_0^6$, $m_1=x_1^3$, $m_2=x_2^3$, $m_3=z_3^2$, $m_4=z_4^2$, with exponents $e=(e_0,\dots,e_4)=(6,3,3,2,2)$, and we let $\zeta_6=e^{2\pi i/6}$ and $\zeta_3=\zeta_6^2$. Groups act on cohomology by pull-back.

\medskip

By \cref{thm:catanese} and adjunction, $K_X=\cO_X(1)$, $K_X^2=1$, $p_g(X)=1$, $q(X)=0$ and $X$ is simply connected. Hence $e(X)=12\chi(\cO_X)-K_X^2=23$, $b_2(X)=21$, $h^{1,1}(X)=19$, and $\Lambda=\operatorname{H}^2(X,\ZZ)$ is a torsion free unimodular lattice of signature $(3,18)$, which is odd since $K^2=1$. By Wu's formula $K$ is a characteristic element\footnote{i.e., $K\cdot \alpha = \alpha^2 \mod{2}$ for every $\alpha\in \operatorname{H}^2(X,\ZZ)$ and, in particular, if $\alpha\in K^\perp$ then $\alpha^2 = 0 \mod 2$.}, so $L=K^\perp$ is even. As $K^2=1$ we have $\Lambda=\ZZ K\oplus L$, so $L$ is unimodular of signature $(2,18)$ and
\[L\simeq \operatorname{II}_{2,18}\simeq U^{\oplus2}\oplus E_8(-1)^{\oplus2}\]
by the classification of indefinite unimodular lattices \cite[Chapter~V]{Ser73}. The polarized Hodge structure on $L$ has Hodge numbers $(1,18,1)$ and polarization $Q=-\langle\ ,\ \rangle$, where $\langle\ ,\ \rangle$ is the intersection form. We also use the Hermitian form $h(x,y)=\int_Xx\wedge\bar y$ on $L\otimes\CC=\operatorname{H}^2(X,\CC)_\prim$, which by the Hodge--Riemann relations is positive definite on $\operatorname{H}^{2,0}\oplus \operatorname{H}^{0,2}$ and negative definite on $\operatorname{H}^{1,1}_\prim$.

\vspace{2mm}

With the same notation as in \cref{sec:extremal}, we let $\tau_1=\sigma_1\varepsilon_1$ and $\tau_2=\sigma_2\varepsilon_2$, that is,
\[\tau_1(x_0,x_1,x_2,z_3,z_4)=(x_0,\zeta_3x_1,x_2,-z_3,z_4),\qquad \tau_2(x_0,x_1,x_2,z_3,z_4)=(x_0,x_1,\zeta_3x_2,z_3,-z_4),\]
so that $G=\langle\tau_1,\tau_2\rangle=\Aut_1(X)\simeq\ZZ/6\ZZ\times\ZZ/6\ZZ$. Fix $\theta\in\CC$ with $\theta^6=r$ and let $\iota$ be the element of the second kind of \cref{prop:family} with $\mu=\theta^3$, $q=-\theta^2$ and $p=-\theta^{-2}$:
\[\iota(x_0,x_1,x_2,z_3,z_4)=(x_0,\,-\theta^2x_2,\,-\theta^{-2}x_1,\,\theta^3z_4,\,\theta^{-3}z_3).\]
Then $\iota^*f=rg$, $\iota^*g=r^{-1}f$, $\iota^2=1$, $\iota\tau_1\iota=\tau_2$ and $\Aut(X)=G\rtimes\langle\iota\rangle$.

\medskip

To describe the primitive cohomology
$\operatorname{H}^2(X,\CC)_{\mathrm{prim}}
=L\otimes_{\ZZ}\CC,$ we use the so-called \emph{Cayley trick}, which associates to the complete intersection
\(X=V(f,g)\) a hypersurface in a projective bundle \cite{Mav99}.
Consider $\cE=\mathcal O_{\PW}(6)^{\oplus2}.$
Thus, we have $\PP(\cE)\simeq\PW\times\PP^1.$

Let \(H\) be the pull-back of the divisor class of
\(\mathcal O_{\PW}(1)\), and let \(F\) be the class
of a fibre of the projection to \(\PP^1\).
We identify \(\operatorname{Cl}(\PP(\cE))\) with \(\ZZ^2\)
using the ordered basis \((F+6H,H)\), and then a bidegree \((a,b)\)
represents the divisor class
$
a(F+6H)+bH=aF+(6a+b)H.
$
The Cox ring of \(\PP(\cE)\) is
$
S:=\CC[x_0,x_1,x_2,z_3,z_4,u_1,u_2].
$
Write \((w_0,\ldots,w_4)=(1,2,2,3,3)\) for the weights of
\((x_0,x_1,x_2,z_3,z_4)\).
The corresponding coordinate divisors have classes \(w_iH\),
whereas \(\{u_j=0\}\) has class \(F=(F+6H)-6H\).
Consequently, the grading is
\[
\deg x_i=(0,w_i)\ (0\leq i\leq2),\qquad
\deg z_i=(0,w_i)\ (i=3,4),\qquad
\deg u_j=(1,-6)\ (j=1,2).
\]
For the Cayley trick, writing $m_0=x_0^6$, $m_1=x_1^3$, $m_2=x_2^3$, $m_3=z_3^2$ and $m_4=z_4^2$, we put
\begin{gather*}\Phi:=u_1f+u_2g=\sum_{i=0}^4\ell_i(u)\,m_i,\\ \ell_0=ru_1+u_2,\qquad \ell_1=u_1+u_2,\qquad \ell_2=-(r^2u_1+u_2),\qquad \ell_3=u_1,\qquad \ell_4=u_2,\end{gather*}
of bidegree $\beta=(1,0)$, which represents the class \(F+6H\) of the ample line bundle
\(\mathcal O_{\PW}(6)
\boxtimes\mathcal O_{\PP^1}(1)\), and where two of the linear forms \(\ell_i\) are linearly independent,
since \(r\notin\{0,\pm1\}\). As the
anticanonical class of \(\PP(\cE)\) is
$
2F+11H,
$
the anticanonical degree is given by \(\beta_0:=(2,-1)\).

The Jacobian ideal of $\Phi$ is
\[J=\Bigl(\tfrac{\partial\Phi}{\partial x_0},\dots,\tfrac{\partial\Phi}{\partial z_4},\tfrac{\partial\Phi}{\partial u_1},\tfrac{\partial\Phi}{\partial u_2}\Bigr)=\bigl(x_0^5\ell_0,\ x_1^2\ell_1,\ x_2^2\ell_2,\ z_3\ell_3,\ z_4\ell_4,\ f,\ g\bigr),\]
and we put $R=S/J$, a bigraded ring. The hypersurface $Y:=\{\Phi=0\}\subset\PP(\cE)$ is quasi-smooth, because $f$ and $g$ have independent gradients at every point of the punctured cone over $X$, and it is ample. Moreover, $Y$ contains $\pi^{-1}(X)=X\times\PP^1$, where $\pi\colon\PP(\cE)\to\PW$ is the projection.

The following theorem, which combines the results of Batyrev and Cox on ample hypersurfaces of toric varieties with a theorem of Mavlyutov, describes the primitive Hodge components of $X$ in terms of the bigraded pieces of $R$. It is the analogue for $X$ of Griffiths' description of the primitive cohomology of a smooth hypersurface by its Jacobian ring \cite{gr69}.

\begin{theorem}[Batyrev--Cox, Mavlyutov]
\label{thm:cayley}
Let $\Omega$ be the toric volume form of $\PP(\cE)$ as in \cite{BC94}. For $q=0,1,2$ the map sending $A\in S_{(q,1)}$ to the $(3-q,1+q)$-component of $\Res(A\Omega/\Phi^{q+2})$ induces an isomorphism $R_{(q,1)}\simeq \operatorname{H}^{3-q,1+q}(Y)_\prim$, and $\psi=j_!\circ\pi^*\colon \operatorname{H}^2(X,\QQ)_\prim\to \operatorname{H}^4(Y,\QQ)_\prim$, where $j\colon X\times\PP^1\to Y$ is the inclusion, is an isomorphism of Hodge structures of type $(1,1)$. Hence
\[A\longmapsto\omega_A=\psi^{-1}\Res\Bigl(\frac{A\,\Omega}{\Phi^{q+2}}\Bigr)^{3-q,1+q}\]
induces isomorphisms $R_{(q,1)}\simeq \operatorname{H}^{2-q,q}(X)_\prim$, and $\dim R_{(q,1)}=1,18,1$ for $q=0,1,2$.
\end{theorem}

\begin{proof}
The first assertion is the theorem of Batyrev and Cox \cite{BC94} for the ample quasi-smooth hypersurface \(Y\), in the form of \cite[Theorem~5.1]{VL24}, applied with \(m=5\) and \(p=3-q\).
The condition \(p\neq m/2-1\) is automatic. For \(A\in S_{(q,1)}\), the form \(A\Omega/\Phi^{q+2}\)
has bidegree zero because
$
(q+2)\beta-\beta_0=(q,1).
$
Taking the \((3-q,1+q)\)-component of its residue gives
the required isomorphism. The second assertion is Mavlyutov's theorem \cite[Theorem~3.6]{Mav99} in the form of \cite[Theorem~5.2 and Remark~5.2]{VL24}, where the primitive cohomology of $X$ used there is the kernel of the Gysin map $\operatorname{H}^2(X)\to \operatorname{H}^6(\PW)$, which is $K^\perp$ because $K$ generates the image of $\operatorname{H}^2(\PW)$ in $\operatorname{H}^2(X)$ and $K^2=1$.
\end{proof}

\medskip

We now determine how the group $G$ acts on the classes $\omega_A$ of \cref{thm:cayley}, which will show that residues of monomials are eigenvectors of $G$. For a monomial $A=x_0^{c_0}x_1^{c_1}x_2^{c_2}z_3^{c_3}z_4^{c_4}u_1^{b_1}u_2^{b_2}$ put
\[k(A)=(k_1,k_2)\in(\ZZ/6\ZZ)^2,\qquad k_1=2(c_1+1)+3(c_3+1),\quad k_2=2(c_2+1)+3(c_4+1).\]
The polynomial $\Phi$ is $G$-invariant, $G=\langle \tau_1,\tau_2\rangle$ acting trivially on $u_1,u_2$. Up to a nonzero constant, $\Omega$ is the contraction of $dx_0\wedge\dots\wedge du_2$ by the two Euler vector fields of the bigrading \cite[Proposition~7.2]{VL24}, which are preserved by every coordinate change multiplying each variable by a scalar. Thus, $\tau_j^*\Omega=\det(\tau_j)\,\Omega=\zeta_6^5\,\Omega$ for $j=1,2$, the determinant being taken on the seven variables, and
\begin{equation}
\label{eq:char}
\tau_1^*\omega_A=\zeta_6^{k_1(A)}\,\omega_A,\qquad \tau_2^*\omega_A=\zeta_6^{k_2(A)}\,\omega_A .
\end{equation}
Thus residues of monomials are eigenvectors of the diagonal automorphisms. In what follows, we write $H_k\subset \operatorname{H}^2(X,\CC)_\prim$ for the eigenspace of the character $k\in(\ZZ/6\ZZ)^2$.

The ideal $J$ is homogeneous for the grading of $S$ by the exponents modulo $e = (e_0,e_1,e_2,e_3,e_4) = (6,3,3,2,2)$. For a monomial $A$ let $a_i\in\{0,\dots,e_i-1\}$ be the residue of $c_i$ modulo $e_i$, write $a=(a_0,\dots,a_4)$ for its class and put $I(a):=\{i\in\{0,\ldots,4\}:a_i=e_i-1\}$. In $R_{(q,1)}$ the degree condition $\sum_iw_ic_i-6(b_1+b_2)=1$ gives $a_0\equiv1-2a_1-2a_2-3a_3-3a_4\pmod 6$, so $a$ and $k(A)$ determine each other, and we also write $H_a$ for $H_{k(A)}$. For example, $x_0$ has class $(1,0,0,0,0)$ and character $(5,5)$, and the monomials $x_0^4x_1x_2\,m_ju^b$ of bidegree $(2,2)$ which appear in \cref{lem:socle} have class $(4,1,1,0,0)$ and character $(1,1)$, and in both cases $I(a)=\varnothing$.

Since $J$ is homogeneous for this grading, the ring $R$ splits as the direct sum of its class-$a$ parts. The following lemma describes each of them as a quotient of a polynomial ring in $m_0,\dots,m_4,u_1,u_2$, which only depends on the set $I(a)$. This reduces the computation of the eigenspaces $H_a$, and of the pairing between them, to linear algebra in these quotients.

\begin{lemma}
\label{lem:iso}
For each subset \(I\subseteq\{0,\ldots,4\}\), define
\[
A_I:=
\frac{\CC[m_0,\ldots,m_4,u_1,u_2]}
{\bigl(f,g,\ m_i\ell_i\ (i\notin I),\ \ell_i\ (i\in I)\bigr)}.
\]
Let \(a=(a_0,\ldots,a_4)\) be a residue class modulo \(e\),
with \(0\leq a_i<e_i\). Multiplication by the monomial
\(x^a=x_0^{a_0}x_1^{a_1}x_2^{a_2}z_3^{a_3}z_4^{a_4}\)
identifies \(A_{I(a)}\), as a vector space, with the class-\(a\)
part of \(R\), denoted \(x^a A_{I(a)}\).

\end{lemma}

\begin{proof}
Put \(B=\CC[m_0,\ldots,m_4,u_1,u_2]\).
The class-\(a\) part of \(S\) is \(x^aB\). A monomial multiple of the generator \(x_i^{e_i-1}\ell_i\) of $J$  lies in class $a$ precisely when it has the form
\(x^a m^\delta u^b\ell_i\), with \(\delta_i\geq1\)
if \(i\notin I(a)\), and \(\delta_i\geq0\) otherwise.
Since \(J\) is homogeneous and \(f,g\) have class zero,
the relations on the coefficient of \(x^a\) are exactly
those defining \(A_{I(a)}\). 
%By homogeneity, this gives \[J\cap x^aB=x^a\bigl(f,g,\ m_i\ell_i\ (i\notin I(a)),\\ell_i\ (i\in I(a))\bigr)B.\] Taking the quotient proves the assertion.
\end{proof}

Combining \cref{thm:cayley} with \cref{lem:iso}, we can now compute the Hodge numbers of each eigenspace $H_a$. Only sixteen classes contribute: fourteen of them give one-dimensional eigenspaces of type $(1,1)$, and the remaining two give the three-dimensional eigenspaces containing $\operatorname{H}^{2,0}(X)$ and $\operatorname{H}^{0,2}(X)$.

\begin{proposition}
\label{prop:dims}
Let \(a=(a_0,\ldots,a_4)\), with \(0\leq a_i<e_i\), satisfy
\[
|a|:=a_0+2a_1+2a_2+3a_3+3a_4\equiv1\pmod6,
\] The dimensions of the Hodge components of \(H_a\), namely
\[
\left(
\dim_{\CC}(H_a\cap \operatorname{H}^{2,0}(X)),
\dim_{\CC}(H_a\cap \operatorname{H}^{1,1}(X)),
\dim_{\CC}(H_a\cap \operatorname{H}^{0,2}(X))
\right),
\]
are as follows:
\begin{enumerate}
\item
\((0,0,0)\) if \(|I(a)|\geq2\).

\item
\((0,1,0)\) if \(|I(a)|=1\).
There are exactly \(14\) such classes, all with \(|a|=7\).

\item
\((1,2,0)\) for \(a=(1,0,0,0,0)\), and
\((0,2,1)\) for \(a=(3,1,1,0,0)\).
These are the only classes with \(I(a)=\varnothing\).
\end{enumerate}
\end{proposition}

\begin{proof}
Put \(m:=(|a|-1)/6\in\ZZ_{\geq0}\).
By \cref{thm:cayley} and \cref{lem:iso}, for \(q=0,1,2\) we have
\[
H_a\cap \operatorname{H}^{2-q,q}(X)\simeq
(A_{I(a)})_{(q-m,q)},
\]
where the first index is the total degree in the variables \(m_i\)
and the second is the total degree in \(u_1,u_2\).
We eliminate \(m_3,m_4\) using \(f,g\), so that their images are
the linear forms
\[
\lambda_3=-(rm_0+m_1-r^2m_2),
\qquad
\lambda_4=-(m_0+m_1-m_2).
\]
Writing \(\lambda_i=m_i\) for \(i\leq2\), any three of
\(\lambda_0,\ldots,\lambda_4\) are linearly independent.

\smallskip
\noindent
(1) Suppose that \(|I(a)|\geq2\).
Two linearly independent forms \(\ell_i\) vanish in \(A_{I(a)}\),
so \(u_1=u_2=0\) there.
Thus only \(q=0\) can contribute, and the nonnegativity of
the first degree then forces \(m=0\).
But \(m=0\) implies \(a=(1,0,0,0,0)\), for which
\(I(a)=\varnothing\), a contradiction.

\smallskip
\noindent
(2) For a singleton \(I=\{s\}\), the quotient
\(\CC[u_1,u_2]/(\ell_s)\) is a polynomial ring \(\CC[t]\),
where \(t\) is the image of any linear form in \(u_1,u_2\)
not proportional to \(\ell_s\).
Each \(\ell_i\) with \(i\neq s\) becomes a nonzero multiple of \(t\).
The four forms \(\lambda_i\) with \(i\neq s\) span the space
of linear forms in \(m_0,m_1,m_2\), so
\[
A_I\simeq
\frac{\CC[m_0,m_1,m_2,t]}
     {(tm_0,tm_1,tm_2)},
\]
with \(\deg m_i=(1,0)\) and \(\deg t=(0,1)\).
Consequently,
\[
\dim_{\CC}(A_I)_{(q-m,q)}
=
\begin{cases}
1,&q=m,\\
0,&q\neq m.
\end{cases}
\]
Applying this to \(I=I(a)\), it remains to determine \(m\).
Listing the classes with \(|I(a)|=1\) and
\(|a|\equiv1\pmod6\) gives the \(14\) classes of
\cref{tab:basis}, all with \(|a|=7\), hence \(m=1\).
Their Hodge dimensions are therefore \((0,1,0)\).

\smallskip
\noindent
(3) Suppose that \(I(a)=\varnothing\).
The conditions \(a_i\leq e_i-2\) and
\(|a|\equiv1\pmod6\) leave only
\[
a=(1,0,0,0,0),\qquad m=0,
\]
and
\[
a=(3,1,1,0,0),\qquad m=1.
\]
In \(A_\varnothing\), the component of bidegree \((0,0)\)
is \(\CC\), and that of bidegree \((0,1)\) is
\(\CC u_1\oplus\CC u_2\).
In bidegree \((1,1)\), the five tensors
\(\lambda_i\otimes\ell_i\) span a subspace of dimension \(4\)
of the \(6\)-dimensional space of bilinear forms:
the first four are linearly independent, whereas
\[
\sum_{i=0}^4\lambda_i\otimes\ell_i=0,
\]
since this sum is the image of \(\Phi=u_1f+u_2g\)
after eliminating \(m_3,m_4\).
Thus
\[
\dim_{\CC}(A_\varnothing)_{(1,1)}=2.
\]
Moreover,
\(\dim_{\CC}(A_\varnothing)_{(1,2)}=1\)
by \cref{lem:socle} below.
This component already contributes one dimension to
\(\operatorname{H}^{0,2}(X)\) for the class \(a=(3,1,1,0,0)\).
Since \(h^{0,2}(X)=1\), it follows that
\((A_\varnothing)_{(2,2)}=0\).
Together with the vanishing of components of negative degree,
these computations give \((1,2,0)\) and \((0,2,1)\),
respectively.
\end{proof}

Choosing a monomial in each of the one-dimensional pieces of \cref{prop:dims}, and suitable polynomials in the two three-dimensional ones, we obtain the elements $A_1,\dots,A_{20}$ listed in \cref{tab:basis}. In terms of $k$, the characters occurring in $\operatorname{H}^2(X,\CC)_\prim$ are
\[(\{1,3,5\}^2\smallsetminus\{(3,3)\})\ \cup\ (\{2,4\}\times\{1,5\})\ \cup\ (\{1,5\}\times\{2,4\}),\]
each of order $6$, with $\dim H_{(5,5)}=\dim H_{(1,1)}=3$ and all other eigenspaces of dimension one. The $(2,0)$-form lies in $H_{(5,5)}$, so $\tau_1$ and $\tau_2$ act on $\operatorname{H}^{2,0}$ by $\zeta_6^{-1}$: indeed $R_{(0,1)}$ is spanned by $x_0$, so $\operatorname{H}^{2,0}$ is spanned by $\omega_{x_0}$ by \cref{thm:cayley}, and $k(x_0)=(5,5)$, so that $\tau_j^*\omega_{x_0}=\zeta_6^5\omega_{x_0}$ by \eqref{eq:char}.

\begin{table}[t]
\centering
\renewcommand{\arraystretch}{1.1}
\begin{tabular}{rlccccl}
\toprule
$i$ & $A_i$ & $k(A_i)$ & type & $I(a)$ & $\iota^*$ & partner index $i^*$,\ \ $c_i$\\
\midrule
1 & $x_0$ & $(5,5)$ & $(2,0)$ & $\varnothing$ & $-1$ & $6$,\ \ $1$ \\
2 & $rx_0^7u_1+2x_0x_1^3u_1$ & $(5,5)$ & $(1,1)$ & $\varnothing$ & $-1$ & $4$,\ \ $2r^3(1-r)$\\
3 & $x_0^7u_1$ & $(5,5)$ & $(1,1)$ & $\varnothing$ & $+1$ & $5$,\ \ $2r^2(1+r)$\\
4 & $x_0^3x_1x_2(ru_1+u_2)$ & $(1,1)$ & $(1,1)$ & $\varnothing$ & $-1$ & $2$\\
5 & $x_0^3x_1x_2(ru_1-u_2)$ & $(1,1)$ & $(1,1)$ & $\varnothing$ & $+1$ & $3$\\
6 & $x_0^3x_1x_2^4u_1^2$ & $(1,1)$ & $(0,2)$ & $\varnothing$ & $-1$ & $1$\\
\midrule
7 & $x_0^5x_1u_1$ & $(1,5)$ & $(1,1)$ & $\{0\}$ & $\to 8$ & $8$,\ \ $r(r+1)$\\
8 & $x_0^5x_2u_1$ & $(5,1)$ & $(1,1)$ & $\{0\}$ & $\to 7$ & $7$\\
9 & $x_0^3x_1^2u_1$ & $(3,5)$ & $(1,1)$ & $\{1\}$ & $\to 11$ & $10$,\ \ $-r^3$\\
10 & $x_0x_1^2x_2u_1$ & $(3,1)$ & $(1,1)$ & $\{1\}$ & $\to 12$ & $9$\\
11 & $x_0^3x_2^2u_1$ & $(5,3)$ & $(1,1)$ & $\{2\}$ & $\to 9$ & $12$,\ \ $1$\\
12 & $x_0x_1x_2^2u_1$ & $(1,3)$ & $(1,1)$ & $\{2\}$ & $\to 10$ & $11$\\
13 & $x_0^4z_3u_2$ & $(2,5)$ & $(1,1)$ & $\{3\}$ & $\to 17$ & $14$,\ \ $r^3(r-1)^2(r+1)$\\
14 & $x_1x_2z_3u_2$ & $(4,1)$ & $(1,1)$ & $\{3\}$ & $\to 18$ & $13$\\
15 & $x_0^2x_1z_3u_2$ & $(4,5)$ & $(1,1)$ & $\{3\}$ & $\to 20$ & $16$,\ \ $r^3(r-1)^2(r+1)$\\
16 & $x_0^2x_2z_3u_2$ & $(2,1)$ & $(1,1)$ & $\{3\}$ & $\to 19$ & $15$\\
17 & $x_0^4z_4u_1$ & $(5,2)$ & $(1,1)$ & $\{4\}$ & $\to 13$ & $18$,\ \ $(r-1)^2(r+1)$\\
18 & $x_1x_2z_4u_1$ & $(1,4)$ & $(1,1)$ & $\{4\}$ & $\to 14$ & $17$\\
19 & $x_0^2x_1z_4u_1$ & $(1,2)$ & $(1,1)$ & $\{4\}$ & $\to 16$ & $20$,\ \ $(r-1)^2(r+1)$\\
20 & $x_0^2x_2z_4u_1$ & $(5,4)$ & $(1,1)$ & $\{4\}$ & $\to 15$ & $19$\\
\bottomrule
\end{tabular}
\caption{The Griffiths basis of $X_r$. The column $\iota^*$ gives the eigenvalue of $\iota^*$ on $\omega_i$ when $I(a)=\varnothing$ or the index $j$ with $\iota^*\omega_i\in\Gm\omega_j$ when $|I(a)|=1$. The constant $c_i=\varphi(A_iA_{i^*})$ of \cref{lem:socle} below is listed once for each pair.}
\label{tab:basis}
\end{table}

In the two three-dimensional eigenspaces we use a basis adapted to $\iota$ (see \cref{prop:iota} below). More precisely, we put
\[A_2:=r\,x_0^7u_1+2\,x_0x_1^3u_1,\quad A_3:=x_0^7u_1,\quad A_4:=x_0^3x_1x_2\,(ru_1+u_2),\quad A_5:=x_0^3x_1x_2\,(ru_1-u_2).\]

The following theorem shows that the residues of these elements form a basis of the primitive cohomology of $X$.

\begin{theorem}
\label{thm:basis}
The $20$ classes $\omega_i=\omega_{A_i}$ of \cref{tab:basis} form a basis of the primitive cohomology $\operatorname{H}^2(X,\CC)_\prim$ consisting of eigenvectors of $G$, with the characters and Hodge types listed there. The classes $\omega_{x_0^7u_1}$ and $\omega_{x_0x_1^3u_1}$ span $\operatorname{H}^{1,1}\cap H_{(5,5)}$, and the classes $\omega_{x_0^3x_1x_2u_1}$ and $\omega_{x_0^3x_1x_2u_2}$ span $\operatorname{H}^{1,1}\cap H_{(1,1)}$.
\end{theorem}

\begin{proof}
By \cref{thm:cayley}, \eqref{eq:char} and \cref{prop:dims} it is enough to check that each listed element of a one-dimensional isotypic piece is nonzero and that the two listed elements of each two-dimensional piece are independent. In the first case the element is $x^au_j$ with $u_j$ not a multiple of $\ell_s$, so its image in $A_{\{s\}}\simeq\CC[m_0,m_1,m_2,t]/\bigl(t\cdot(m_0,m_1,m_2)\bigr)$ is a nonzero multiple of the variable $t$ introduced in the proof of \cref{prop:dims}(2). In the second case the pieces are $\CC u_1\oplus\CC u_2$ and the quotient of the bilinear forms by $\langle\lambda_i\otimes\ell_i\rangle$, in which $m_0\otimes u_1$ and $m_1\otimes u_1$ are independent, as one checks directly.
\end{proof}

\medskip

We now compute the intersection form in the basis of \cref{thm:basis}. All products $A_iA_{i^*}$, where $i^*$ is the partner index of $i$ as in \cref{tab:basis}, lie in the class $(4,1,1,0,0)$ and in bidegree $(2,2)$. They are of the form $x_0^4x_1x_2\,m_j\,u^b$ with $b_1+b_2=2$. These products determine the intersection form: by $G$-invariance, $\langle\omega_A,\omega_B\rangle$ vanishes unless the characters of $A$ and $B$ are inverse to each other, and \cref{prop:cup} below expresses the remaining values through a linear form on the class-$(4,1,1,0,0)$ part of $R_{(2,2)}$, which we now determine.

\begin{lemma}
\label{lem:socle}
The class-$(4,1,1,0,0)$ part of $R_{(2,2)}$ is one-dimensional. The linear form $\varphi$ on it normalized by $\varphi(x_0^4x_1x_2^4u_1^2)=1$ takes the following values on $x_0^4x_1x_2\,m_j\,u^b$:
\begin{center}
\renewcommand{\arraystretch}{1.1}
\begin{tabular}{cccc}
\toprule
$j$ & $u_1^2$ & $u_1u_2$ & $u_2^2$\\
\midrule
$0$ & $r(r+1)$ & $-r^2(r+1)$ & $r^3(r+1)$\\
$1$ & $-r^3$ & $r^3$ & $-r^3$\\
$2$ & $1$ & $-r^2$ & $r^4$\\
$3$ & $0$ & $0$ & $r^3(r-1)^2(r+1)$\\
$4$ & $(r-1)^2(r+1)$ & $0$ & $0$\\
\bottomrule
\end{tabular}
\end{center}
\end{lemma}

\begin{proof}
By \cref{lem:iso} this part is $x^a(A_\varnothing)_{(1,2)}$ with $a=(4,1,1,0,0)$. It is spanned by the $15$ monomials $m_ju^b$, subject to the relations $\ell_jm_ju_k=0$ and $fu^b=gu^b=0$. Put $v_j=\varphi(m_ju_1^2)$. For $j=0,1,2$ the relations $\ell_jm_ju_k=0$ ($k=1,2$) determine $\varphi(m_ju_1u_2)$ and $\varphi(m_ju_2^2)$ from $v_j$. The relations $fu^b=gu^b=0$ give $\varphi(m_3u^b)=\varphi(\lambda_3u^b)$ and $\varphi(m_4u^b)=\varphi(\lambda_4u^b)$. The remaining relations $u_1m_3u_k=0$ and $u_2m_4u_k=0$ then reduce to
\[rv_0+v_1-r^2v_2=0,\qquad r^2v_0+v_1-r^4v_2=0 ,\]
the other two being consequences of these. For $r\notin\{0,\pm1\}$ the solution space is spanned by $(v_0,v_1,v_2)=(r(r+1),-r^3,1)$, which proves one-dimensionality and gives the table.
\end{proof}

The next proposition relates the intersection form of $X$ to the linear form $\varphi$. It follows from the toric version of the Carlson--Griffiths cup product formula due to Villaflor \cite{VL24}, applied to the hypersurface $Y$, together with the comparison between the cohomology of $X$ and that of $Y$ given by the isomorphism $\psi$ of \cref{thm:cayley}.

\begin{proposition}
\label{prop:cup}
There is a constant $\kappa\in\Gm$ such that, for all monomials $A,B\in S_{(1,1)}$,
\[\langle\omega_A,\omega_B\rangle=\begin{cases}\kappa\,\varphi(AB)&\text{if }k(A)+k(B)=0,\\ 0&\text{otherwise.}\end{cases}\]
where $\varphi$ is the linear form of \cref{lem:socle} and character addition is taken in $(\ZZ/6\ZZ)^2$. Moreover, the generators $\omega_1\in \operatorname{H}^{2,0}(X)$ and
$\omega_6\in \operatorname{H}^{0,2}(X)$ satisfy $
\kappa':=\langle\omega_1,\omega_6\rangle\neq0.$
\end{proposition}

\begin{proof}
By $G$-invariance, the intersection vanishes if
$k(A)+k(B)\neq0$. Assume henceforth that $k(A)+k(B)=0$.
Then $AB\in S_{(2,2)}$ has residual class
$a=(4,1,1,0,0)$, since
$k(AB)=k(A)+k(B)-(5,5)=(1,1)$. For such pairs we apply the cup product formula of
\cite[Theorem~1.4]{VL24} to the ample quasi-smooth fourfold $Y=\{\Phi=0\}\subset\PP(\cE)$.
We carry out the computation in toric coordinates for which
$Y$ is non-degenerated, meaning that its intersection with
every torus orbit is a smooth hypersurface or is empty. This property fails in the original coordinates, since $Y$
contains the torus-fixed points
\[
(0:0:0:1:0)\times(0:1),
\qquad
(0:0:0:0:1)\times(1:0).
\]
We therefore choose distinct
$t_1,t_2\in\CC\setminus\{0,-1,-r,-r^2\}$ and introduce
coordinates $v_1,v_2$ by
\[
u_1=v_1+v_2,\qquad u_2=t_1v_1+t_2v_2.
\]
Putting $f'=f+t_1g$ and $g'=f+t_2g$, we obtain
$\Phi=v_1f'+v_2g'$. Our choice of $t_1,t_2$ ensures that
all five pure powers occur in both $f'$ and $g'$ with
nonzero coefficients. We verify non-degeneracy in these coordinates.
Let $T_\sigma$ be a torus orbit of $\PW$.
On $T_\sigma\times\{v_1v_2\neq0\}$, the equation is
$f'+sg'=0$, where $s=v_2/v_1$.
At a singular point, differentiation with respect to $s$
gives $g'=0$, and hence $f'=0$. Since $t_1\neq t_2$,
this implies $f=g=0$. The remaining derivative equations give
a relation between the gradients restricted to the orbit:
\[
\lambda\nabla f+\mu\nabla g=0,
\qquad
(\lambda,\mu)=(1+s,t_1+st_2)\neq(0,0).
\]
To rule this out, lift to the coordinate torus over $T_\sigma$.
The coefficient pairs of $m_0,\ldots,m_4$ in $(f,g)$ are
\[
(r,1),\quad (1,1),\quad (-r^2,-1),\quad (1,0),\quad (0,1),
\]
which are pairwise independent because $r\notin\{0,\pm1\}$.
Since the $m_i$ are pure powers, the derivative equations
therefore force $\lambda=\mu=0$ whenever at least two
weighted coordinates are nonzero, a contradiction.
If only one is nonzero, then $X\cap T_\sigma=\varnothing$. Finally, on $T_\sigma\times\{v_1=0\}$ and
$T_\sigma\times\{v_2=0\}$, the equation reduces to $g'=0$
and $f'=0$, respectively. These are sums of pure powers
with nonzero coefficients, whose zero loci on the orbit
are smooth or empty. Thus $Y$ is non-degenerated with
respect to the coordinates $(x,v)$. We regard $S$ and $\Phi$ as fixed, expressing all their elements,
including $A$ and $B$, in the new coordinates. The chain rule gives
\[
\begin{pmatrix}
\partial\Phi/\partial v_1\\
\partial\Phi/\partial v_2
\end{pmatrix}
=
\begin{pmatrix}
1&t_1\\
1&t_2
\end{pmatrix}
\begin{pmatrix}
\partial\Phi/\partial u_1\\
\partial\Phi/\partial u_2
\end{pmatrix}.
\]
The remaining partial derivatives are unchanged. Since
$t_1\neq t_2$, both coordinate systems therefore give the same
Jacobian ideal $J$ and quotient $R=S/J$.
With compatible normalizations, the toric volume forms satisfy
$\Omega_u=(t_2-t_1)\Omega_v$. Consequently,
\[
\widetilde\omega_A^{\,v}
:=
\operatorname{Res}\left(\frac{A\Omega_v}{\Phi^3}\right)^{2,2},
\qquad
\psi(\omega_A)=\omega_A^{\,u}=(t_2-t_1)\widetilde\omega_A^{\,v}.
\]
Put $M_v=x_0x_1x_2z_3z_4v_1v_2$.
By \cite[Theorem~1.4]{VL24},
$\int_Y\widetilde\omega_A^{\,v}\smile\widetilde\omega_B^{\,v}$
is a fixed nonzero multiple of $c(AB)$, where $c(P)$ is determined by
\[
PM_v\equiv c(P)\operatorname{Hess}_v(\Phi)
\pmod{J_{0,v}(\Phi)},
\qquad P\in S_{(2,2)}.
\]
Here $\operatorname{Hess}_v(\Phi)$ is the toric Hessian \cite[Definition~6.3]{VL24} and $J_{0,v}(\Phi)=(y\,\partial\Phi/\partial y)_y$. These are computed in the coordinates
$y\in\{x_0,x_1,x_2,z_3,z_4,v_1,v_2\}$.
Since $M_v$ is divisible by every coordinate,
$M_vJ\subseteq J_{0,v}(\Phi)$.
Thus $c$ descends to $R_{(2,2)}$, and its restriction to the
class-$a$ component is a scalar multiple of $\varphi$ by \cref{lem:socle}. On the other hand, for $x,x'\in \operatorname{H}^2(X,\QQ)_\prim$ the projection formula gives $\int_Y\psi(x)\cup\psi(x')=\int_{X\times\PP^1}\pi^*(x\cup x')\cup c_1(\mathcal{N})$, where $\mathcal{N}$ is the normal bundle of $X\times\PP^1$ in $Y$. Since $c_1(\mathcal{N})$ has degree $-1$ on the fibres of $\pi$, this is $-\langle x,x'\rangle$. Hence $\kappa$ exists, and it is nonzero because the intersection form is nondegenerate on $\operatorname{H}^{1,1}_\prim$. The last assertion holds because $\operatorname{H}^{2,0}$ pairs nondegenerately with $\operatorname{H}^{0,2}$ and trivially with $\operatorname{H}^{1,1}$.
\end{proof}

We now collect these computations. For $c\in\CC$ we write
\[\hyp{c}=\begin{pmatrix}0&c\\ c&0\end{pmatrix}\]
for the Gram matrix of a pair of isotropic vectors with product $c$. With this notation, the intersection form in the basis of \cref{tab:basis} is as follows.

\begin{theorem}
\label{thm:gram}
Order the basis of \cref{tab:basis} by the ten partner pairs, namely $(\omega_1,\omega_6)$, $(\omega_2,\omega_4)$ and $(\omega_3,\omega_5)$, followed by the pairs $(\omega_{2j+1},\allowbreak\omega_{2j+2})$ with $3\leq j\leq9$. Then the Gram matrix of the intersection form is block diagonal, equal to
\[
\begin{aligned}
\kappa\Bigl[
&\hyp{-\tfrac23}
 \oplus \hyp{2r^3(1-r)}
 \oplus \hyp{2r^2(1+r)}\\
&\oplus \hyp{r(r+1)}
 \oplus \hyp{-r^3}
 \oplus \hyp{1}\\
&\oplus \hyp{r^3(r-1)^2(r+1)}^{\oplus2}
 \oplus \hyp{(r-1)^2(r+1)}^{\oplus2}
\Bigr],
\end{aligned}
\]
where $\kappa$ is the constant of \cref{prop:cup} and the Gram matrix of the polarization $Q$ is its negative. A common rescaling $\omega_i\mapsto\kappa^{-1/2}\omega_i$
removes the factor $\kappa$. In the monomial basis of the two-dimensional pieces, the pairing $\bigl(\operatorname{H}^{1,1}\cap H_{(5,5)}\bigr)\times\bigl(\operatorname{H}^{1,1}\cap H_{(1,1)}\bigr)\to\CC$ is given by
\[\bigl(\langle\omega_{A},\omega_{B}\rangle\bigr)_{A\in\{x_0^7u_1,\ x_0x_1^3u_1\},\ B\in\{x_0^3x_1x_2u_1,\ x_0^3x_1x_2u_2\}}
=\kappa\begin{pmatrix} r(r+1) & -r^2(r+1)\\ -r^3 & r^3\end{pmatrix}.\]
\end{theorem}

\begin{proof}
Hodge types and $G$-invariance leave only the partner pairs of \cref{tab:basis} and, inside $H_{(5,5)}\times H_{(1,1)}$, the $2\times2$ block above. The first block follows from \cref{rem:kappa-relation}. All remaining entries follow from \cref{prop:cup} and \cref{lem:socle}. For instance $A_7A_8=x_0^4x_1x_2\,m_0u_1^2$ and $A_{13}A_{14}=A_{15}A_{16}=x_0^4x_1x_2\,m_3u_2^2$. For the adapted basis one computes $\varphi(A_2A_4)=2(r-1)v_1=2r^3(1-r)$, $\varphi(A_3A_5)=2rv_0=2r^2(1+r)$ and $\varphi(A_2A_5)=\varphi(A_3A_4)=0$.
\end{proof}

The partner index $i^*$ in the last column of \cref{tab:basis} has a natural interpretation in terms of complex conjugation, and the basis is also well adapted to the Hermitian form $h$, as the following corollary shows.

\begin{corollary}
\label{cor:herm}
The Hermitian form $h$ is diagonal in the basis $\omega_1,\dots,\omega_{20}$, with $h(\omega_1,\omega_1)>0$, $h(\omega_6,\omega_6)>0$ and $h(\omega_i,\omega_i)<0$ for $i\neq1,6$. Complex conjugation maps $\CC\omega_i$ onto $\CC\omega_{i^*}$.
\end{corollary}

\begin{proof}
Automorphisms of $X$ preserve $h$, so eigenvectors for distinct characters of $G$, or for distinct eigenvalues of the involution $\iota^*$, are $h$-orthogonal, and the Hodge decomposition is $h$-orthogonal. Inside $\operatorname{H}^{1,1}\cap H_{(5,5)}$ and $\operatorname{H}^{1,1}\cap H_{(1,1)}$ the eigenspaces of $\iota^*$ are lines by \cref{prop:iota}. Conjugation maps $H_k$ onto $H_{-k}$, $\operatorname{H}^{p,q}$ onto $\operatorname{H}^{q,p}$ and preserves the eigenspaces of $\iota^*$. The signs are the Hodge--Riemann relations.
\end{proof}

\medskip

Finally, we use the action of $G$ to decompose the rational primitive cohomology into pieces defined over $\QQ$. The eigenvalues of $G$ on $L\otimes\CC$ are sixth roots of unity, and complex conjugation, which generates the Galois group of $\QQ(\zeta_6)$ over $\QQ$, maps the eigenspace $H_k$ onto $H_{-k}$, because $\overline{\zeta_6^{\,j}}=\zeta_6^{-j}$. Since the action of \(G\) is defined over \(\QQ\), each eigenspace \(H_k\) is defined over \(\QQ(\zeta_6)\). Hence each sum $H_k\oplus H_{-k}$ is defined over $\QQ$, and the isotypic components of the representation of $G$ on $L\otimes\QQ$ are the intersections of these sums with $L\otimes\QQ$, since every character of \cref{tab:basis} has order six, $k\neq-k$, and the pairs $\{k,-k\}$ are the partner pairs of \cref{tab:basis}. We obtain an orthogonal decomposition
\begin{equation}
\label{eq:rational}
L\otimes\QQ=V_1\oplus V_2\oplus\dots\oplus V_8 ,
\end{equation}
where $V_1\otimes\CC=H_{(5,5)}\oplus H_{(1,1)}=\langle\omega_1,\dots,\omega_6\rangle$ has Hodge numbers $(1,4,1)$ and $V_j\otimes\CC=\langle\omega_{2j+3},\omega_{2j+4}\rangle$ is of type $(1,1)$ for $j=2,\dots,8$. The pieces $V_2,\dots,V_8$ are $\QQ$-irreducible.

The piece $V_1$ is not of type $(1,1)$, but it can be refined by means of the involution $\iota$, which does not belong to $G$. The following proposition describes the action of $\iota$ on the basis of \cref{tab:basis}. In particular, it explains the choice of the elements $A_2,\dots,A_5$ in the two three-dimensional eigenspaces.

\begin{proposition}
\label{prop:iota}
Extend $\iota$ to $S$ by $\iota^*u_1=r^{-1}u_2$ and $\iota^*u_2=ru_1$, so that $\iota^*\Phi=\Phi$. Then $\iota^*$ acts by $-1$ on $\omega_1,\omega_2,\omega_4,\omega_6$ and by $+1$ on $\omega_3,\omega_5$, and it permutes the lines $\CC\omega_7,\dots,\CC\omega_{20}$ as indicated in \cref{tab:basis}. In particular $\iota$ preserves $V_1$ and $V_2$ and exchanges $V_3\leftrightarrow V_4$, $V_5\leftrightarrow V_7$ and $V_6\leftrightarrow V_8$.
\end{proposition}

\begin{proof}
Since $\iota^*\Phi=\Phi$, we have $\iota^*\omega_A=\det(\iota)\,\omega_{\iota^*A}$, as in \eqref{eq:char}. The seven variables split into $x_0$, which $\iota$ fixes, and three pairs of variables which $\iota$ exchanges up to scalars with product $1$, so $\det(\iota)=(-1)^3=-1$. In $R$ one has $\iota^*A_1=A_1$, $\iota^*A_2=A_2$, $\iota^*A_3=-A_3$, $\iota^*A_4=A_4$, $\iota^*A_5=-A_5$ and $\iota^*A_6=A_6$, using the relations $x_0m_0\ell_0\in J$ and $x_0m_2\ell_2\in J$: for instance $\iota^*(x_0^7u_1)=r^{-1}x_0m_0u_2\equiv-x_0^7u_1$, and $\iota^*A_6=-r^{-3}x_0^3x_1x_2\,m_1u_2^2\equiv A_6$, because the class-$(3,1,1,0,0)$ part of $R_{(2,1)}$ is $x^a(A_\varnothing)_{(1,2)}$ and $\varphi(m_1u_2^2)=-r^3$ by \cref{lem:socle}. The permutation of the remaining lines follows from $\iota\tau_1\iota=\tau_2$, which gives $\iota^*H_{(k_1,k_2)}=H_{(k_2,k_1)}$.
\end{proof}

By \cref{prop:iota}, the involution $\iota^*$ preserves $V_1$ and acts on it with eigenvalues $\pm1$. Its two eigenspaces separate the classes $\omega_3,\omega_5$, which are of type $(1,1)$, from the part of $V_1$ containing $\operatorname{H}^{2,0}(X)$ and $\operatorname{H}^{0,2}(X)$. We fix notation for them, which will be used throughout the rest of the paper.

\begin{notation}
\label{not:V1}
Let $V_1^{\tr}$ and $V_1^{\alg}$ be the eigenspaces of $\iota^*$ on $V_1$ for the eigenvalues $-1$ and $+1$, so that $V_1^{\tr}\otimes\CC=\langle\omega_1,\omega_2,\omega_4,\omega_6\rangle$ has Hodge numbers $(1,2,1)$, and $V_1^{\alg}\otimes\CC=\langle\omega_3,\omega_5\rangle$ is of type $(1,1)$. 
\end{notation}

We now let $r$ vary and study the infinitesimal variation of the Hodge structure of the family $\{X_r\}$. The following lemma shows that the period map of this family is an immersion. This will be used in \cref{cor:rho} to show that, for very general $r$, the Hodge structure of $X_r$ has no extra algebraic classes.

\begin{lemma}
\label{lem:ivhs}
Let $\nabla$ be the Gauss--Manin connection of the family $\{X_r\}$. Then $\nabla_{d/dr}\,\omega_1\equiv\frac2r\,\omega_2$ modulo $F^2\operatorname{H}^2(X_r,\CC)$. In particular the period map $r\mapsto[\operatorname{H}^{2,0}(X_r)]$ is an immersion.
\end{lemma}

\begin{proof}
Since the isomorphism $\psi=j_!\circ\pi^*$ of \cref{thm:cayley} is defined topologically, it commutes with $\nabla$, and by Griffiths transversality the image of $\nabla_{d/dr}\omega_1$ in $\operatorname{H}^{1,1}$ is the class of the $(1,1)$-part of the residue of $\frac{d}{dr}(x_0\Omega/\Phi^2)=-2x_0\,\partial_r\Phi\,\Omega/\Phi^3$ \cite{gr69,CG80}, that is, $-2\,\omega_{x_0\partial_r\Phi}$, where $\partial_r\Phi=u_1(m_0-2rm_2)$. Since $x_0u_1f\in J$ and $x_0u_1m_3=x_0z_3\cdot z_3u_1\in J$, we get $r^2x_0m_2u_1\equiv rx_0m_0u_1+x_0m_1u_1$ in $R$, hence $x_0\,\partial_r\Phi\equiv-\frac1r\bigl(rx_0^7u_1+2x_0x_1^3u_1\bigr)=-\frac1rA_2$, which is nonzero in $R_{(1,1)}$. 
\end{proof}

\begin{remark}\label{rem:kappa-relation}
It follows that the constants of \cref{prop:cup} satisfy
$\kappa'=-\frac23\kappa$.
Indeed, the same argument from \cref{lem:ivhs} allows us to conclude 
$\nabla_{d/dr}\omega_1\equiv(2/r)\omega_2\pmod{F^2}$ and
$\nabla_{d/dr}\omega_4
\equiv 6r^2(1-r)\omega_6
\pmod{F^1}.$ Differentiating $\langle\omega_1,\omega_4\rangle=0$,
flatness of the intersection form and Hodge types yield
\[
\begin{aligned}
0
&=\frac2r\langle\omega_2,\omega_4\rangle
  +6r^2(1-r)\langle\omega_1,\omega_6\rangle\\
&=r^2(1-r)(4\kappa+6\kappa'),
\end{aligned}
\]
where the last equality follows from \cref{prop:cup} and \cref{lem:socle}. Since $r\neq0,1$, the relation follows.
\end{remark}

We can now bound the Picard number of $X_r$. The rational pieces $V_1^{\alg},V_2,\dots,V_8$ consist of algebraic classes, so the transcendental part of the cohomology of $X_r$ is contained in $V_1^{\tr}$. Combined with \cref{lem:ivhs}, this determines the Picard number of $X_r$ for very general $r$, a notion that will be used throughout \cref{sec:hodge}.

\begin{corollary}
\label{cor:rho}
Let $T(X_r)$ be the transcendental lattice of $X_r$, the orthogonal complement of $\NS(X_r)$ in $\Lambda = \operatorname{H}^2(X,\ZZ)$, and let $V_1^{\tr}$ and $V_1^{\alg}$ be as in \cref{not:V1}. Then the following holds.
\begin{enumerate}[(1)]
\item $\NS(X_r)_\QQ\supseteq\QQ K\oplus V_1^{\alg}\oplus V_2\oplus\dots\oplus V_8$. Hence $\rho(X_r)\geq17$ and $T(X_r)_\QQ\subseteq V_1^{\tr}$.
\item $\tau_1$ makes $V_1^{\tr}$ a $\QQ(\zeta_6)$-vector space of dimension $2$ and $T(X_r)_\QQ$ a subspace. In particular, $\rho(X_r)\in\{17,19\}$.
\item For $r$ outside a countable subset of $\Gm\smallsetminus\{\pm1\}$ one has $\rho(X_r)=17$ and $\End_\Hdg(T(X_r)_\QQ)=\QQ(\zeta_6)$. We say that $r$, or $X_r$, is \emph{very general} if these two conditions hold.
\end{enumerate}
\end{corollary}

\begin{proof}
Put $T_r=T(X_r)_{\QQ}$ and $F=\QQ(\zeta_6)$.
By the Lefschetz $(1,1)$-theorem, the rational spaces
$V_1^{\alg},\allowbreak V_2,\allowbreak\ldots,\allowbreak V_8$ consist of algebraic classes.
Together with $\QQ K$, they form a $17$-dimensional subspace
whose orthogonal complement is $V_1^{\tr}$.
Thus $\rho(X_r)\geq17$ and $T_r\subseteq V_1^{\tr}$.
The minimal polynomial $t^2-t+1$ of $\tau_1$ makes
$V_1^{\tr}$ an $F$-vector space of dimension two.
Since $\tau_1$ preserves both $\operatorname{NS}(X_r)$ and the
intersection form, $T_r$ is an $F$-subspace. It is nonzero
because $T_r\otimes\CC$ contains $\operatorname{H}^{2,0}(X_r)$,
so $\dim_{\QQ}T_r\in\{2,4\}$. As $b_2(X_r)=21$,
this gives $\rho(X_r)\in\{17,19\}$. The Hodge structure $T_r$ is simple, since
$h^{2,0}(T_r)=1$ and it has no nonzero rational classes of
type $(1,1)$. Hence, by \cite{Zar83},
$E=\End_\Hdg(T_r)$ is a number field
containing $F$, and
\[
\dim_{\QQ}T_r=[E:\QQ]\dim_E T_r.
\]
We now work on the universal cover of the parameter space.
There we can choose $\theta=r^{1/6}$ holomorphically, so $\iota$
defines an automorphism of the pulled-back family.
The actions of $G$ and $\iota$ commute with parallel transport, and we may therefore identify $V_1^{\tr}$ with a fixed
rational vector space, with a fixed $F$-action. If $\rho(X_r)=19$, then $T_r$ is an $F$-line in
$V_1^{\tr}$. The line $\operatorname{H}^{2,0}(X_r)$ lies in
the $\zeta_6^{-1}$-eigenspace, since it has character
$(5,5)$, and hence equals that eigenspace. Thus this $F$-line determines the period. If instead $\rho(X_r)=17$ and $E\neq F$, then
$T_r=V_1^{\tr}$, $[E:\QQ]=4$, and $\dim_E T_r=1$.
Consequently,
\[
V_1^{\tr}\otimes_{\QQ}\CC
=\bigoplus_{\sigma:E\hookrightarrow\CC}V_\sigma,
\qquad
\dim_{\CC}V_\sigma=1,
\]
where $\sigma$ runs over the four embeddings of $E$ into $\CC$ and $V_\sigma$ is the line on which $E$ acts through $\sigma$. Since $E$ preserves $\operatorname{H}^{2,0}(X_r)$, this is one of these
four eigenlines. In either exceptional case, the period therefore belongs to
a fixed countable set of lines. Indeed, there are countably
many $F$-lines in $V_1^{\tr}$ and countably many quartic
subfields of $\operatorname{End}_{\QQ}(V_1^{\tr})$:
by the primitive element theorem, each such field is generated
by one rational endomorphism.
By \cref{lem:ivhs} the period map is an immersion, so the inverse
image of each fixed line is discrete and hence countable.
The union of these inverse images, and its image in the
original parameter space, is countable. Outside this set,
$\rho(X_r)=17$ and
$\operatorname{End}_{\mathrm{Hdg}}(T_r)=F$, as required.
\end{proof}

%----------------------------------------------------------------------
\section{Hodge isometries of $X_r$}
\label{sec:hodge}
%----------------------------------------------------------------------

In this section we compare the automorphism group of $X_r$ with the group $\Ort_\Hdg(L)$ of Hodge isometries of $L$, using the results of \cref{sec:griffiths}. To prove \cref{thm:B}, we construct many Hodge isometries of $L$ which are not induced by automorphisms, namely reflections in roots of the N\'eron--Severi lattice orthogonal to $K$. For this, we first study the sublattices of $L$ cut out by the rational decomposition \eqref{eq:rational}. We introduce them in the following notation, where we also recall some standard facts on discriminant forms of even lattices that will be used to decide when isometries of these sublattices extend to $L$.

\begin{notation}
\label{not:hodge}
We keep the notation of \cref{sec:griffiths}. In particular, $X=X_r$ and $L=K^\perp\subseteq\operatorname{H}^2(X,\ZZ)$, with intersection form $\langle\ ,\ \rangle$, and $\Aut(X_r)=G\rtimes\langle\iota\rangle$, where $G=\langle\tau_1,\tau_2\rangle$. We also keep the basis $\omega_1,\dots,\omega_{20}$ and the eigenspaces $H_k$ of \cref{thm:basis}, the rational subspaces $V_1,\dots,V_8$ of \eqref{eq:rational} and the subspaces $V_1^{\tr},V_1^{\alg}$ of \cref{not:V1}. We define the following primitive sublattices of $L$:
\[
T_0=L\cap V_1^{\tr},\qquad
N_0=T_0^\perp\cap L,\qquad
M_0=L\cap V_1^{\alg},\qquad
L_j=L\cap V_j\quad(2\leq j\leq8).
\]
For a nondegenerate even lattice $M$, we write $M^*$ for its dual, $\cA_M=M^*/M$ for its discriminant group, $q_M$ for its discriminant form and $\Ort(M)^\sharp$ for the kernel of $\Ort(M)\to\Ort(q_M)$. For a prime $p$, we say that $\cA_M$ is \emph{$p$-elementary} if $p\,\cA_M=0$. If $R$ and $R'$ are primitive sublattices of $M$ which are orthogonal complements of each other, the finite group $M/(R\oplus R')$ is called the \emph{glue group}. Finally, we write $A_2(m)$ for the lattice with Gram matrix $m\left(\begin{smallmatrix}2&-1\\-1&2\end{smallmatrix}\right)$, and $D_4=\{x\in\ZZ^4:\sum x_i\equiv0\pmod2\}$.
\end{notation}

Since $T_0\otimes\QQ=V_1^{\tr}$ has dimension four and Hodge numbers $(1,2,1)$, the lattice $T_0$ has rank $4$ and signature $(2,2)$ by the Hodge--Riemann relations. Thus $N_0$ has rank $16$ and is negative definite, while $M_0,L_2,\ldots,L_8$ are pairwise orthogonal even sublattices of $N_0$ of rank $2$. For very general $r$ in the sense of \cref{cor:rho}, these lattices satisfy $T_0=T(X_r)$ and $N_0=\NS(X_r)\cap K^\perp$.

We also recall the following standard facts on discriminant forms. With $R,R'\subseteq M$ as above, the orthogonal projections to $R^*$ and $R'^*$ induce injections of the glue group $M/(R\oplus R')$ into $\cA_R$ and $\cA_{R'}$. When $M$ is unimodular, both injections are isomorphisms, and the image of the glue group in $\cA_R\oplus\cA_{R'}$ is the graph of an isomorphism
\[
\gamma\colon\cA_R\longrightarrow\cA_{R'},
\qquad q_{R'}\circ\gamma=-q_R.
\]
Moreover, a pair of isometries $(g,g')\in\Ort(R)\times\Ort(R')$ extends to an isometry of $M$ if and only if $\gamma\circ\bar g=\bar g'\circ\gamma$, where $\bar g$ and $\bar g'$ denote the induced automorphisms of the discriminant groups \cite[Corollary~1.5.2]{Nik80}. Finally, $\cA_{A_2(m)}\simeq\ZZ/m\ZZ\times\ZZ/3m\ZZ$ and $D_4$ has discriminant $4$.

Before comparing $\Aut(X_r)$ with $\Ort_\Hdg(L)$, we check that the natural map between them is injective, and that it remains injective after passing to the quotient by $\pm1$. This is needed to compare the orders of these groups in \cref{thm:B}.

\begin{proposition}
\label{prop:faithful}
The action of $\Aut(X_r)$ on $L$ is faithful, and
$-\mathrm{id}_L$ is not induced by an automorphism
of $X_r$.
\end{proposition}

\begin{proof}
We first consider an element $\tau_1^m\tau_2^n\in G$,
which acts on $H_k$ by $\zeta_6^{mk_1+nk_2}$.
If it acts trivially on $L$, the characters $(1,1)$
and $(1,2)$ in \cref{tab:basis} give $
m+n\equiv m+2n\equiv0\pmod6.
$ It follows that $m\equiv n\equiv0\pmod6$, so the
element is the identity. Suppose now that $\tau_1^m\tau_2^n$ acts as
$-\mathrm{id}_L$. Then $mk_1+nk_2\equiv3\pmod6$
for every character in \cref{tab:basis}.
Using again $(1,1)$ and $(1,2)$, we obtain
$n\equiv0$ and $m\equiv3\pmod6$.
But the character $(2,1)$ then gives
$2m+n\equiv0\pmod6$, a contradiction. It remains to consider the elements $g\iota$ with
$g\in G$. By \cref{prop:iota}, the involution $\iota$
maps $\CC\omega_7$ onto $\CC\omega_8$, while $g$
preserves every eigenline. Thus $g\iota$ also maps
$\CC\omega_7$ onto $\CC\omega_8\neq\CC\omega_7$,
so it acts neither trivially nor as $-\mathrm{id}_L$.
\end{proof}

To prove \cref{thm:B}, we identify root lattices in $N_0$ and use their root reflections to construct Hodge isometries of $L$. To this end, we rely on the following three lemmas. They are probably standard, but we include their proofs for completeness. The first one is used systematically in the study of automorphisms of K3 surfaces and of hyper-K\"ahler manifolds (see for instance \cite{BC23}), and the other two are elementary instances of the correspondence between lattices with a fixed-point-free isometry of order three and Hermitian lattices over the Eisenstein integers (see for instance \cite[Chapter~2, \S2.6]{CS99}).

\begin{lemma}
\label{lem:pelem}
Let $M$ be a unimodular lattice and $\sigma\in\Ort(M)$ of prime order $p$. Then $\cA_{M^\sigma}$ is $p$-elementary, that is, $p\,\cA_{M^\sigma}=0$.
\end{lemma}

\begin{proof}
Let $y\in(M^\sigma)^*$. Since $M$ is unimodular and $M^\sigma$ is primitive, the orthogonal projection $M\to(M^\sigma)^*$ is surjective, so $y$ is the projection of some $z\in M$, that is, $y=\frac1p\sum_{i=0}^{p-1}\sigma^iz$. Hence $py\in M^\sigma$.
\end{proof}

\begin{lemma}
\label{lem:hex}
Let $M$ be an even lattice of rank $2$ with an isometry of order $3$ without nonzero fixed vectors. Then $M\simeq A_2(m)$ for a unique $m\in\ZZ\smallsetminus\{0\}$.
\end{lemma}

\begin{proof}
The isometry allows us to regard $M$ as a torsion free module of rank $1$ over $\ZZ[\zeta_3]$. Since this ring is a principal ideal domain, the module is free, so we may identify $M$ with $\ZZ[\zeta_3]$. Let $b$ denote the bilinear form of $M$ and write $b(1,1)=2m$, with $m\in\ZZ$ since $M$ is even. Multiplication by $\zeta_3$ preserves $b$, so $b(\zeta_3,\zeta_3)=2m$. The relation $1+\zeta_3+\zeta_3^2=0$ then gives $b(1,\zeta_3)=-m$. Thus the basis $1,\zeta_3$ has the Gram matrix of $A_2(m)$, as required.
\end{proof}

\begin{lemma}
\label{lem:eis}
Let $M$ be an even lattice of rank $4$ with an isometry $\sigma$ of order $3$ without nonzero fixed vectors. Then $\disc M$ is a perfect square. If moreover $M$ is negative definite and $\cA_M$ is $2$-elementary, then $M\simeq D_4(-1)$.
\end{lemma}

\begin{proof}
Let $b$ denote the bilinear form of $M$. Put $F=\QQ(\zeta_3)$ and $\mathcal O=\ZZ[\zeta_3]$, and write $\Tr$ and $\Nm$ for the trace and norm of $F/\QQ$. As in the preceding lemma, $M$ is a free $\mathcal O$-module of rank two. To compute its discriminant, consider the Hermitian form $H(x,y)=\frac13\sum_{k=0}^2\bar\zeta_3^{\,k}b(x,\zeta_3^k y)$, which satisfies $b=\Tr\circ H$. Since $\Tr(\lambda H(x,y))=b(x,\lambda y)\in\ZZ$ for every $\lambda\in\mathcal O$, the values of $H$ lie in $\frac1{\sqrt{-3}}\mathcal O$. Moreover, $H(x,x)=\frac12b(x,x)\in\ZZ$, because $M$ is even. Thus, in an $\mathcal O$-basis, its matrix has the form
\[
H=
\begin{pmatrix}
\alpha & \beta/\sqrt{-3}\\
-\bar\beta/\sqrt{-3} & \gamma
\end{pmatrix},
\qquad
\alpha,\gamma\in\ZZ,\quad \beta\in\mathcal O,
\]
and $3\det H=3\alpha\gamma-\Nm(\beta)\in\ZZ$. The trace pairing $(x,y)\mapsto\Tr(x\bar y)$ on $\mathcal O$ has determinant $3$. Passing to the associated $\ZZ$-basis and taking the trace of $H$, we obtain $\disc M=3^2\Nm(\det H)=(3\det H)^2$, which proves the first assertion. Assume now that $M$ is negative definite and $\cA_M$ is $2$-elementary. We first show that $\disc M=4$. Since $\cA_M$ has at most four generators, $\disc M=2^t$ with $0\leq t\leq4$, so the first assertion gives $\disc M\in\{1,4,16\}$. As $H$ is negative definite, $d:=3\det H>0$, and hence $d\in\{1,2,4\}$. Writing $\beta=p+q\zeta_3$, we have $\Nm(\beta)=p^2-pq+q^2\equiv(p+q)^2\not\equiv2\pmod3$. Together with $d=3\alpha\gamma-\Nm(\beta)$, this excludes $d=1$ and $d=4$. Thus $d=2$ and $\disc M=4$. To identify $M$, we first construct a sublattice isometric to $A_2(-1)$. The minimum of the positive definite lattice $M(-1)$ is at most $\gamma_4(\disc M)^{1/4}=2$, where $\gamma_4=\sqrt2$ is the Hermite constant \cite[Chapter~1]{CS99}. Since $M$ is even, there is $e_1\in M$ with $b(e_1,e_1)=-2$. Put $e_2=\sigma e_1$. The relation $1+\sigma+\sigma^2=0$ gives $b(e_1,e_2)=1$, so $R=\ZZ e_1+\ZZ e_2\simeq A_2(-1)$. Its discriminant is squarefree, hence $R$ is primitive. The orthogonal complement $N=R^\perp\cap M$ is $\sigma$-stable, so \cref{lem:hex} gives $N\simeq A_2(-m)$ for some $m\geq1$. Write $k=[M:R\oplus N]$. The injection of the glue group into $\cA_R$ gives $k\mid3$, while $9m^2=\disc(R\oplus N)=k^2\disc M=4k^2$. It follows that $k=3$ and $m=2$. It remains to describe how $R$ and $N$ are glued inside $M$. Choose a standard basis $f_1,f_2$ of $N\simeq A_2(-2)$. The classes of $u=\frac{2e_1+e_2}{3}$ and $v=\frac{f_1+2f_2}{3}$ generate $\cA_R$ and the $3$-torsion subgroup $\cA_N[3]$, respectively, both cyclic of order three. The glue group projects isomorphically onto each of these groups, so it is generated by $(\bar u,\bar v)$ or $(\bar u,-\bar v)$. After replacing $f_1,f_2$ by their negatives if necessary, we may therefore write
\[
M=(R\oplus N)+\ZZ h,
\qquad
h=u+v=\frac{2e_1+e_2+f_1+2f_2}{3}.
\]
A direct computation shows that $e_1,e_2,-h,f_2-h$ have Gram matrix
\[
\begin{pmatrix}
-2&1&1&1\\
1&-2&0&0\\
1&0&-2&0\\
1&0&0&-2
\end{pmatrix}.
\]
This is the Gram matrix of $D_4(-1)$, and its determinant is $4=\disc M$. Hence these vectors form a basis of $M$, giving the required isometry.
\end{proof}

We now apply these lemmas to the action of $G$ on $L$. The invariant lattices of the elements $\tau_1^2,\tau_2^2$ of order three and $\tau_1^3,\tau_2^3$ of order two are root lattices contained in $N_0$, and the reflections in their roots will give the Hodge isometries of \cref{thm:B}.

\begin{lemma}
\label{prop:roots}
The lattices $L_3$ and $L_4$ are isometric to $A_2(-1)$, while the invariant lattices $L^{\tau_1^3}=L\cap(V_5\oplus V_6)$ and $L^{\tau_2^3}=L\cap(V_7\oplus V_8)$ are isometric to $D_4(-1)$. In particular, $N_0$ contains the pairwise orthogonal root sublattices $L_3$, $L_4$, $L^{\tau_1^3}$ and $L^{\tau_2^3}$, of type $A_2\sqcup A_2\sqcup D_4\sqcup D_4$.
\end{lemma}

\begin{proof}
We first consider the isometry $\tau_1^2$ of order $3$ on $L$. By \cref{tab:basis}, the characters fixed by $\tau_1^2$ are precisely those with $k_1=3$, which occur only in $V_3$. Thus its invariant lattice is $L\cap V_3=L_3$. Since $L$ is unimodular, \cref{lem:pelem} shows that $\cA_{L_3}$ is $3$-elementary. On the other hand, $\tau_2^2$ acts on $L_3$ as an isometry of order $3$ without nonzero fixed vectors. As $L_3$ is even and negative definite, \cref{lem:hex} gives $L_3\simeq A_2(-m_3)$ for some $m_3\geq1$. Its discriminant group is $\ZZ/m\ZZ_3\times\ZZ/3m\ZZ_3$, which is $3$-elementary only for $m_3=1$. Hence $L_3\simeq A_2(-1)$. The same argument, exchanging $\tau_1$ and $\tau_2$, gives $L_4\simeq A_2(-1)$.

We now consider the involution $\tau_1^3$, which acts on $H_k$ by $(-1)^{k_1}$. Its invariant lattice is $L\cap(V_5\oplus V_6)$, an even negative definite lattice of rank $4$. By \cref{lem:pelem}, its discriminant group is $2$-elementary. Moreover, $\tau_1$ restricts to an isometry of order $3$ without nonzero fixed vectors, since $k_1\in\{2,4\}$ on $V_5\oplus V_6$. We can therefore apply \cref{lem:eis} to obtain $L^{\tau_1^3}\simeq D_4(-1)$. The same argument for $\tau_2^3$ gives $L^{\tau_2^3}\simeq D_4(-1)$.

Finally, the four lattices lie respectively in $V_3$, $V_4$, $V_5\oplus V_6$ and $V_7\oplus V_8$, which are pairwise orthogonal. These subspaces are also orthogonal to $V_1$, and hence to $T_0\subseteq V_1$. Thus the four lattices are contained in $N_0$, proving the last assertion.
\end{proof}

We can now prove \cref{thm:B}. The reflections in the roots of the four lattices of \cref{prop:roots} generate a large group $W$ of Hodge isometries of $L$, and we compare it with the image of $\Aut(X_r)$ in $\Ort_\Hdg(L)$.

\begin{proof}[Proof of \cref{thm:B}]
Let $v\in N_0$ be a root, so $v^2=-2$. The reflection $
s_v(x)=x+\langle x,v\rangle v $ is an integral isometry of $L$ fixing $T_0$ pointwise. Since $v$ is orthogonal to
$T_0\otimes\CC\supseteq\operatorname{H}^{2,0}(X_r)$,
the reflection $s_v$ fixes $\operatorname{H}^{2,0}(X_r)$
pointwise. Being defined over $\QQ$, it also fixes
$\operatorname{H}^{0,2}(X_r)
=\overline{\operatorname{H}^{2,0}(X_r)}$ pointwise.
As an isometry, it therefore preserves
$
\operatorname{H}^{1,1}(X_r)_{\mathrm{prim}}
=
\bigl(\operatorname{H}^{2,0}(X_r)
\oplus\operatorname{H}^{0,2}(X_r)\bigr)^\perp. $ Thus $s_v\in\Ort_\Hdg(L)$. Let $W\subseteq\Ort_\Hdg(L)$ be the group generated by the reflections in the roots of the four lattices of \cref{prop:roots}. Since these lattices are pairwise orthogonal, $W$ acts on their sum as the product of their Weyl groups and trivially on its orthogonal complement, so
\[
W\simeq W(A_2)^2\times W(D_4)^2,
\qquad
|W|=6^2\cdot192^2=1\,327\,104.
\]
Moreover, $W$ acts trivially on their orthogonal complement
in $L\otimes\QQ$, namely $V_1\oplus V_2$. The four lattices consist of classes of type $(1,1)$, so they lie in $\NS(X_r)$ by the Lefschetz $(1,1)$ theorem, and they are orthogonal to $K$. This proves the first assertion. By \cref{prop:faithful}, we identify
$A=\Aut(X_r)=G\rtimes\langle\iota\rangle$
with its image in $\Ort_\Hdg(L)$. Recall that $|A|=72$.
The group $G$ preserves each of the four root lattices,
while $\iota$ exchanges $L_3$ with $L_4$ and
$L^{\tau_1^3}$ with $L^{\tau_2^3}$.
Since conjugation carries $s_v$ to $s_{a(v)}$ for $a\in A$,
the group $A$ normalizes $W$. We claim that $
A\cap W=\{1,\tau_1^3\tau_2^3\}.$ Indeed,
an element of the form $g\iota$, with $g\in G$, exchanges
$L_3$ and $L_4$, whereas every element of $W$ preserves
each of them. Thus $A\cap W\subseteq G$.
If $\tau_1^m\tau_2^n\in W$, its action on $V_1\oplus V_2$
is trivial. The characters of these two spaces give
\[
m+n\equiv m-n\equiv0\pmod6,
\]
hence $(m,n)=(0,0)$ or $(3,3)$ modulo $6$.
Conversely, \cref{tab:basis} shows that $\tau_1^3\tau_2^3$ acts
as $-\mathrm{id}$ on both $D_4(-1)$ lattices and as the
identity on their orthogonal complement.
Since $-\mathrm{id}\in W(D_4)$, it belongs to $W$. Therefore $WA$ is a subgroup of $\Ort_\Hdg(L)$ of order
\[
|WA|=\frac{|W|\,|A|}{|W\cap A|}
=36\,|W|=47\,775\,744.
\]
In particular,
\[
[\Ort_\Hdg(L):A]\geq[WA:A]
=\frac{|W|}{2}=663\,552.
\]
Finally, $-\tau_1^3$ acts as $-\mathrm{id}$ on
$L^{\tau_1^3}$ and as the identity on its orthogonal
complement, so $-\tau_1^3\in W$ and hence $-\mathrm{id}_L\in WA$.
The image of $WA$ in $\Ort_\Hdg(L)/\{\pm1\}$ therefore
has order $18\,|W|$. By \cref{prop:faithful}, the image
$\overline A$ of $A$ in this quotient still has order $72$.
Consequently,
\[
[\Ort_\Hdg(L)/\{\pm1\}:\overline A]
\geq\frac{18\,|W|}{72}
=\frac{|W|}{4}=331\,776.
\]
The result follows.
\end{proof}

We now compare the actions of $\Aut(X_r)$ and $\Ort_\Hdg(L)$ on the transcendental lattice. For very general $r$ in the sense of \cref{cor:rho}, this lattice is $T_0=T(X_r)$, while $N_0$ is the orthogonal complement of the canonical class in $\NS(X_r)$. The following proposition determines the Hodge isometries of $T_0$, and the image and the kernel of the restriction map from $\Aut(X_r)$.

\begin{proposition}
\label{prop:kernel}
Let $r$ be very general. Then the group of Hodge isometries of $T_0$ is $\Ort_\Hdg(T_0)=\mu_6=\langle\tau_1|_{T_0}\rangle$. Moreover, the restriction map $\pi\colon\Aut(X_r)\to\Ort_\Hdg(T_0)$ is surjective, and its kernel is the \emph{dicyclic group} $\operatorname{Dic}_{12}$ of order $12$ generated by $\delta=\tau_1\tau_2^{-1}$ and $j=\tau_1^3\iota$, with $j^2=\delta^3$ and $j\delta j^{-1}=\delta^{-1}$.
\end{proposition}

\begin{proof}
The Hodge structure $T=T_0\otimes\QQ$ is irreducible and $E=\End_\Hdg(T)=\QQ(\zeta_6)$ by assumption. A Hodge isometry $g$ of $T_0$ lies in $E^\times$, and it preserves the positive definite real plane $(\operatorname{H}^{2,0}\oplus \operatorname{H}^{0,2})\cap T_\RR$ and its orthogonal complement, so $\Ort_\Hdg(T_0)$ is contained in a compact group. Being discrete, it is finite. A finite order element of $\QQ(\zeta_6)^\times$ is a sixth root of unity, so $\Ort_\Hdg(T_0)\subseteq\mu_6$, and $\tau_1|_{T_0}$, which acts on $\operatorname{H}^{2,0}$ by $\zeta_6^{-1}$, generates $\mu_6$. In particular $\pi$ is surjective.

On $T_0\otimes\CC=\langle\omega_1,\omega_2,\omega_4,\omega_6\rangle$ the element $\tau_1^m\tau_2^n$ acts by $\zeta_6^{-s}$ on the $(5,5)$-part and by $\zeta_6^{s}$ on the $(1,1)$-part, where $s=m+n$. So it lies in $\ker\pi$ if and only if $s\equiv0\pmod6$, which gives the cyclic group $\langle\delta\rangle$ of order $6$. By \cref{prop:iota}, $\iota$ acts as $-1$ on $T_0$, so $\tau_1^m\tau_2^n\iota$ lies in $\ker\pi$ if and only if $s\equiv3\pmod6$. This gives six further elements, among them $j$. Finally $\iota\tau_1\iota=\tau_2$ gives $j\delta j^{-1}=\delta^{-1}$ and $j^2=\tau_1^3\tau_2^3=\delta^3$.
\end{proof}

To pass from $T_0$ to $L$, we also need the action on $N_0$. Since $L$ is unimodular and $T_0$ and $N_0$ are orthogonal complements of each other, a Hodge isometry of $L$ is determined by its restrictions to $T_0$ and $N_0$, which must induce compatible automorphisms of the discriminant groups, as recalled after \cref{not:hodge}. This gives the following description of $\Ort_\Hdg(L)$.

\begin{proposition}
\label{prop:glue}
Let $r$ be very general. Restriction to $T_0$ gives an exact
sequence
\[
1\longrightarrow\Ort(N_0)^\sharp
\longrightarrow\Ort_\Hdg(L)
\longrightarrow\mu_6
\longrightarrow1.
\]
Identifying $\Aut(X_r)$ with its image in $\Ort_\Hdg(L)$,
we have $
\Aut(X_r)\cap\Ort(N_0)^\sharp
=\ker\pi\simeq\operatorname{Dic}_{12},
$ where $\pi$ is the restriction map to $T_0$.
Moreover, $
|\Ort(N_0)^\sharp|\geq6\,|W|=7\,962\,624.
$
\end{proposition}

\begin{proof}
A Hodge isometry of $L$ preserves
$L\cap\operatorname{H}^{1,1}(X_r)=N_0$, hence also its
orthogonal complement $T_0$ in $L$. It therefore restricts
to a pair
$(g_T,g_N)\in\Ort_\Hdg(T_0)\times\Ort(N_0)$.
Since $L$ is unimodular and $T_0$, $N_0$ are primitive
orthogonal complements, the glue group $L/(T_0\oplus N_0)$
is the graph of an isomorphism
$\gamma\colon\cA_{T_0}\to\cA_{N_0}$ satisfying
$q_{N_0}\circ\gamma=-q_{T_0}$.
By \cref{not:hodge}, a pair $(g_T,g_N)$ of isometries
extends to an isometry of $L$ if and only if
$\gamma\circ\bar g_T=\bar g_N\circ\gamma$.
Since $N_0$ is of type $(1,1)$, this extension preserves
the Hodge decomposition whenever $g_T$ does. The restriction map
$\Ort_\Hdg(L)\to\Ort_\Hdg(T_0)=\mu_6$
is surjective because $\tau_1|_{T_0}$ generates $\mu_6$.
Its kernel consists of the pairs $(\mathrm{id},g_N)$
satisfying $\bar g_N=\mathrm{id}$, and is therefore
identified with $\Ort(N_0)^\sharp$.
Its intersection with $\Aut(X_r)$ is precisely $\ker\pi$,
which is isomorphic to $\operatorname{Dic}_{12}$
by \cref{prop:kernel}.
Since $N_0$ is negative definite, $\Ort(N_0)^\sharp$ is
finite, and the exact sequence gives
$|\Ort_\Hdg(L)|=6\,|\Ort(N_0)^\sharp|$.
The claimed bound follows from \cref{thm:B}.
\end{proof}

For very general $r$, the preceding propositions show
that the action on the transcendental lattice identifies
a quotient $\ZZ/6\ZZ$ of $\Aut(X_r)$ with
$\Ort_\Hdg(T_0)$. The difference in \cref{thm:B}
therefore comes from the action on $N_0$: every isometry
of $N_0$ acting trivially on $\cA_{N_0}$ extends uniquely
to a Hodge isometry of $L$ fixing $T_0$ pointwise. Automorphisms of $X_r$ also preserve the ample cone. The natural continuation of \cref{thm:B} is therefore the study of Hodge isometries preserving this cone. To formulate this question, recall that
$\Lambda=\ZZ K\oplus L$, so every
$g\in\Ort_\Hdg(L)$ extends uniquely to an isometry
$\widetilde g$ of $\Lambda$ fixing $K$. Put
\[
\Ort^{+}_\Hdg(L)
=
\{g\in\Ort_\Hdg(L):
\widetilde g(\operatorname{Amp}(X_r))
=\operatorname{Amp}(X_r)\}.
\]
Since every automorphism of $X_r$ fixes $K$ and
preserves the ample cone, we have $
\Aut(X_r)\subseteq\Ort^{+}_\Hdg(L).
$
We do not know whether this inclusion is an equality.
Such an equality would give, for automorphisms of $X_r$,
an analogue of the strong Torelli theorem for K3
surfaces \cite{burns1975torelli}.
Catanese proved that the global period map for surfaces
with $K^2=p_g=1$ is not injective
\cite[Theorem~3.3]{Cat80}. This result concerns the
determination of a surface by its periods, and does
not settle the question above for a fixed surface $X_r$. Our computations with the Jacobian ring do not determine
the ample cone. Describing this cone requires information
about the classes of irreducible curves of negative
self-intersection and the integral structure of $N_0$.
The latter is only partially described by
\cref{prop:roots}, which identifies the root sublattice $
A_2(-1)^{\oplus2}\oplus D_4(-1)^{\oplus2}
\subseteq N_0.
$

%----------------------------------------------------------------------
%\section*{Acknowledgments}
%----------------------------------------------------------------------
%
%ACKNOWLEDGMENTS.

\bibliographystyle{alpha}
\bibliography{ref}

\end{document}